\documentclass[11pt,reqno]{amsart}
\numberwithin{equation}{section}
\usepackage[utf8]{inputenc}
\usepackage{amsmath, enumerate, graphicx, mathrsfs, multicol, fullpage, amssymb, hyperref, url, fancyref}
\usepackage{extarrows, musicography, comment, appendix, mathscinet }
\usepackage[dvipsnames]{xcolor}
\usepackage{breqn}

\def\XXint#1#2#3{{\setbox0=\hbox{$#1{#2#3}{\int}$ }
\vcenter{\hbox{$#2#3$ }}\kern-.6\wd0}}

\newcommand{\weakto}{\rightharpoonup}
\newcommand{\n}{\hat{\text{n}}}
\newcommand{\bbA}{\mathbb{A}}
\newcommand{\DD}{\mathbb{D}}

\newcommand{\RR}{\mathbb{R}}

\newcommand{\PP}{\mathbb{P}}
\newcommand{\Div}{\nabla\cdot}
\newcommand{\Lbf}{\mathbf{L}}
\newcommand{\Hbf}{\mathbf{H}}
\newcommand{\Jbf}{\mathbf{J}}
\newcommand{\Nn}{\mathcal{N}}
\newcommand{\NN}{\mathbb{N}}
\newcommand{\Ee}{\mathcal{E}}
\newcommand{\calH}{\mathcal{H}}

\newcommand{\parf}{\partial_{x_1}}
\newcommand{\pars}{\partial_{x_2}}
\newcommand{\parff}{\partial_{x_1x_1}}
\newcommand{\parfs}{\partial_{x_1x_2}}
\newcommand{\parss}{\partial_{x_2x_2}}

\renewcommand{\SS}{\mathbb{S}}
\renewcommand{\div}{\text{div}}

\DeclareMathOperator{\diam}{diam}
\DeclareMathOperator{\dist}{dist}
\DeclareMathOperator{\GL}{GL}

\DeclareMathOperator{\supp}{supp}

\DeclareMathOperator{\Id}{\mathbb{I}}

\DeclareMathOperator{\curl}{curl}

\newtheorem{theorem}{Theorem}[section]
\newtheorem{definition}[theorem]{Definition}

\newtheorem{lemma}[theorem]{Lemma}
\newtheorem{proposition}[theorem]{Proposition}
\newtheorem{remark}[theorem]{Remark}
\numberwithin{equation}{section}

\title{Global Existence of Weak Solutions to the Two Dimensional Nematic Liquid Crystal Flow with Partially Free Boundary}

\author[Y. Sire]{Yannick Sire}
\address{\noindent
Department of Mathematics,
Johns Hopkins University, 3400 N. Charles Street, Baltimore, MD 21218, USA}
\email{ysire1@jhu.edu}

\author[Y. Wu]{Yantao Wu}
\address{\noindent
Department of Mathematics,
Johns Hopkins University, 3400 N. Charles Street, Baltimore, MD 21218, USA}
\email{ywu212@jhu.edu}

\author[Y. Zhou]{Yifu Zhou}
\address{\noindent
School of Mathematics and Statistics, Wuhan University, Wuhan 430072, China}
\email{yifuzhou@whu.edu.cn}

\date{\today}
\begin{document}

\maketitle
\begin{abstract}
We consider a nematic liquid crystal flow with partially free boundary in a smooth bounded domain in $\RR^2$. We prove regularity estimates and the global existence of weak solutions enjoying partial regularity properties, and a uniqueness result.
\end{abstract}

\tableofcontents

\section{Introduction and main results}

We consider the following hydrodynamic system modeling the flow of liquid crystal materials in dimension two:

\begin{equation}\label{LCF}
\begin{cases}
u_t + u\cdot\nabla u -\nu \Delta u + \nabla P  = -\lambda\nabla\cdot(\nabla d\odot\nabla d - \frac12|\nabla d|^2\Id_2)  \\
\nabla\cdot u  = 0  \\
d_t + u\cdot \nabla d  = \gamma(\Delta d+|\nabla d|^2d)
\end{cases}  \text{in } \Omega\times(0,T)
\end{equation}
where  domain $\Omega\subseteq\RR^2$ is assumed to be a connected  bounded domain with boundary $\Gamma:=\partial\Omega$ of class $C^\infty$. Function $u(x,t):\Omega\times(0,+\infty)\to\RR^2$ represents the velocity field of the flow,   $d(x,t):\Omega\times(0,+\infty)\to\SS^2\subseteq\RR^3$ is a unit-vector which represents the director field of macroscopic molecular orientation of the liquid crystal material, and $P(x,t):\Omega\times(0,+\infty)\to\RR$ represents the pressure function. Three positive constants $\nu,~\lambda$, and $\gamma$ respectively quantifies viscosity,   the competition between kinetic energy and elastic energy,   microscopic elastic relaxation time for the director field. Here we assume that $\nu=\lambda=\gamma=1$ since the exact values of these constants play no role in our qualitative results.  $\Div$ denotes the divergence operator, and $\nabla d\odot\nabla d$ denotes the $2\times2$ matrix whose $(i,j)$-th entry is given by $\partial_{x_i}d\cdot\partial_{x_j}d$ for $i,j\in\{1,2\}$.

\medskip

The system \eqref{LCF} was first proposed by Lin in \cite{L1989CPAM}, and it is a simplified version of the Ericksen-Leslie system (\cite{Ericksen1962ARMA}, \cite{Leslie1968ARMA}). In fact, these models both share the same type energy law, coupling structure and dissipative properties. There have been growing interests concerning the global existence of weak solutions, partial regularity results, singularity formation and others. 
Lin and Liu \cite{LL1995CPAM} established the global existence of weak and 
classical solutions in two and three dimensions. A partial regularity result of Caffarelli-Kohn-Nirenberg type (\cite{CKN1982CPAM}) for suitable weak solutions was showed in \cite{LL1996DCDS}. In two dimensions, the global existence of Leray-Hopf-Struwe type weak solutions of \eqref{LCF} was proved in Lin-Lin-Wang \cite{LLW10}, and the uniqueness of such weak solution was later shown in \cite{Lin-Wang10}. See also \cite{LCF2D} for the construction of finite-time singularities. Much less is known in three dimensions due to the super-criticality. In \cite{LW2016CPAM}, Lin and Wang proved the global existence of suitable weak solutions that satisfy the global energy inequality, under the assumption that the initial orientation field $d_0(\Omega)\subset \mathbb{S}^2_+$. There are also blow-up examples and criteria for finite-time singularities, for instance, \cite{HW2012CPDE,HLLW2016ARMA}. We refer to a comprehensive survey by Lin and Wang \cite{LW2014RSTA} for recent important developments of mathematical analysis of nematic liquid crystals.

\medskip

In this paper, our concern is a free boundary model of \eqref{LCF} recently introduced in \cite{LSWZ23}. We consider the system \eqref{LCF} with partially free boundary conditions 
\begin{equation*} 
\begin{cases}
u\cdot \n = 0   \\
(\DD u\cdot \n)_\tau = 0   \\
d(x,t) \in \Sigma  \\
\nabla_{\n} d(x,t) \perp T_{d(x,t)}\Sigma  \\
\end{cases}  \text{on } \partial\Omega\times(0,T)
\end{equation*}
where $\n$ and $\tau$ are the unit outer normal vector and tangential vector of $\partial\Omega$,   and $\DD u$ is deformation tensor associated with the velocity field $u$
\[\DD u=\frac12(\nabla u+(\nabla u)^T),\]
and $\Sigma\subset\SS^2$ is a simple, closed and smooth curve. 
The case that $\Sigma$ is a circle in $\SS^2$ is physically relevant.  First two boundary conditions  are the usual \emph{Navier perfect-slip boundary conditions} for the velocity field, indicating a zero friction along the boundary, and last two are referred to as \emph{partially free boundary conditions} for the map. This boundary conditions are physically natural in that it  agrees with \emph{the basic energy law} \eqref{app: basic energy law}. Let us assume that $\Sigma$ is the equator for simplicity, then the partially free boundary condition can be simplified to be  
\begin{equation}\label{PFB}
\begin{cases}
(\DD u\cdot\n)_\tau = u\cdot\n=0   \\
\nabla_{\n}d_1 = \nabla_{\n}d_2 = d_3 =0
\end{cases}
\text{ on } \partial\Omega\times(0,T).
\end{equation}

\medskip

The free boundary system \eqref{LCF}-\eqref{PFB} is a nonlinearly coupled system between the incompressible Navier-Stokes equations and the harmonic map heat flow with a partially free boundary condition. The latter is a geometric flow with the Plateau and Neumann type boundary conditions. The motivation for studying this model is from a recent surge of interests in geometric variational problems and the Navier-Stokes equation with Navier boundary condition. The former includes an interesting and classical topic of the harmonic map heat flow with free boundary, and we refer to Hamilton \cite{Hamilton1975LNM},   Struwe \cite{StruweManMath1991}, Chen-Lin \cite{ChenLin98JGA}, Ma \cite{MaLiCMH1991}, Sire-Wei-Zheng \cite{sire2019singularity}, Hyder-Segatti-Sire-Wang \cite{HSSW} and the references therein for both seminal and more recent results. On the other hand, the Navier-Stokes equation with Navier boundary condition is more physical in some contexts. See, for instance, \cite{Constantin-Foias,Tartar-Book,daVeiga2006CPAA,Kelliher06,daVeiga2009,Neustupa2010,Chen2010Indiana,AR2014,SiranLi2019,AACG2021} and their references. As derived and discussed in \cite{LSWZ23}, the system \eqref{LCF} turns out to be physically natural and compatible with the free boundary condition \eqref{PFB} imposed, and it enjoys the same dissipative properties as those for the case of Dirichlet boundary. Moreover, \eqref{LCF}-\eqref{PFB} may trigger new boundary behaviors of solutions, such as the finite-time blow-up via bubbling on the boundary (and in the interior) constructed in \cite{LSWZ23}. This was already observed by Chen and Lin \cite{ChenLin98JGA} in the context of harmonic map heat flow with free boundary.

\medskip

In this paper, we are interested in the global existence of weak solutions, partial regularity and uniqueness of \eqref{LCF}-\eqref{PFB}. The well-posedness under consideration is motivated by the interesting work of Lin, Lin and Wang \cite{LLW10} which considers the no-slip boundary condition for velocity field $v$ and Dirichlet boundary condition for director field $d$. 

\medskip

 Let us first define
\begin{align*}
 \Lbf &  = \{u\in L^2(\Omega,\RR^2) : \nabla\cdot u =0 \text{ in } \Omega \text{ and } v \text{ satisfies } \eqref{PFB}_1 \}, \\
 \Hbf  & = \{u\in H^1(\Omega,\RR^2) : \nabla\cdot u=0 \text{ in } \Omega \text{ and } v \text{ satisfies } \eqref{PFB}_1 \}, \\
\Jbf  & = \{ d\in H^1(\Omega,\RR^3): d\in\SS^2\text{ a.e. in } \Omega \text{ and } d \text{ satisfies  } \eqref{PFB}_2    \}.
\end{align*}
We assume that the initial data 
\begin{equation}\label{IC}  
   (u(x,0), d(x,0))   = (u_0(x),d_0(x)), \ \  x\in\Omega
 \end{equation}
for \eqref{LCF}-\eqref{PFB} 
satisfy
\begin{equation}\label{assumption1}
u_0 \in \Lbf \text{ and } d_0\in \Jbf.
\end{equation}

Since $\Delta u = \nabla \cdot (\nabla u + (\nabla u)^T)=2\nabla\cdot \DD u$ and partially free boundary condition gives $\n^T(\DD u)\tau = \n(\nabla d\odot \nabla d)\tau=0$, the following weak formulation can help us get rid of boundary terms:
\begin{definition}[Weak Solution] For $0<T\leq+\infty, u\in L^\infty([0,T],\Lbf)\cap L^2([0,T],\Hbf)$ and $d\in L^2([0,T], \Jbf ))$ is a weak solution of \eqref{LCF}-\eqref{IC}, if
\begin{align*}
&-\int_{\Omega\times[0,T]} \langle u,\psi'\varphi\rangle + \int_{\Omega\times[0,T]} [\langle u\cdot\nabla u, \psi\varphi\rangle +\langle2 \DD u,\psi\nabla\varphi\rangle] \\
&= -\psi(0)\int_{\Omega}\langle u_0,\varphi\rangle +  \int_{\Omega\times[0,T]}\langle \nabla d\odot\nabla d - \frac12|\nabla d|^2\Id_2,\psi\nabla\varphi\rangle, \\
& -\int_{\Omega\times[0,T]} \langle d,\psi'\phi\rangle + \int_{\Omega\times[0,T]} [\langle u\cdot\nabla d, \psi\phi\rangle + \langle \nabla d,\psi\nabla\phi\rangle] \\
&= -\psi(0)\int_{\Omega}\langle d_0,\phi\rangle +  \int_{\Omega\times[0,T]}  |\nabla d|^2 \langle d,\psi \phi\rangle  
\end{align*}
for any $\psi\in C^\infty([0,T])$ with $\psi(T)=0$, $\varphi\in H^1 (\Omega,\RR^2)$ with $\varphi\cdot\n=0$ on $\partial\Omega$, and $\phi\in H^1 (\Omega,\RR^3)$ with $\phi_3=0$  on $\partial\Omega$. Moreover, $(u,d)$ satisfies \eqref{assumption1} in the sense of trace.  
\end{definition}

Our first regularity theorem for the weak solution is stated as follows.

\begin{theorem}\label{Thm A} For $0<T<+\infty$,  assume   $u\in L^\infty([0,T],\Lbf)\cap L^2([0,T],\Hbf)$ and 
$d\in L^2([0,T],  \Jbf)$ is a weak solution of \eqref{LCF}-\eqref{IC} with initial data satisfying \eqref{assumption1}. If $d$ further belongs to $L^2([0,T], H^2(\Omega))$, then $(u,d)\in C^\infty(\Omega\times (0,T])\cap C^{2,1}_\alpha(\overline{\Omega}\times(0,T])$ for some $\alpha\in(0,1)$.
\end{theorem} 

Our second theorem concerns the existence of global weak solutions that enjoy the partial smoothness property, and a uniqueness result.

\begin{theorem}\label{Thm B} There exists a global weak solution $u\in L^\infty([0,\infty),\Lbf)\cap L^2([0,+\infty), \Hbf)$ and $d\in L^\infty([0,\infty), H^1(\Omega,\SS^2))$ of \eqref{LCF}-\eqref{IC} with \eqref{assumption1}, such that the following properties hold:
\begin{enumerate}
\item There exists $L\in\NN$ depending only on $(u_0,d_0)$ and $0<T_1<\cdots<T_L,~ 1\leq i\leq L$, such that
\[ (u,d)\in C^\infty(\Omega\times((0,\infty)\setminus\{T_i\}_{i=1}^L)) \cap C_\alpha^{2,1}(\overline{\Omega}\times((0,+\infty)\setminus\{T_i\}_{i=1}^L) ).\] 

\item Global weak solution $(u,d)$ is unique in the class of functions that
\[ d \in L^\infty([0,\infty), H^1(\Omega)) \bigcap_{i=0}^{K-1} \cap_{\epsilon>0} L^2([T_i, T_{i+1}-\epsilon], H^2(\Omega)) \cap L^2([T_K,+\infty), H^2(\Omega))  \]
for some $0<T_1<\cdots<T_K<+\infty$.

\item Each singular time $T_i$ ($1\leq i\leq L$) can be characterized by
\begin{equation}\label{Singular Time}
\liminf_{t\nearrow T_i} \max_{x\in\overline\Omega} \int_{\Omega\cap B_r(x)} (|u|^2 + |\nabla d|^2)(y,t) dy \geq 4\pi, \ \forall r>0.
\end{equation}
Moreover, there exist $x_m^i\to x_0^i \in\overline{\Omega}, ~t_m^i\nearrow T_i,~ r_m^i \searrow 0$ such that 
\begin{enumerate}
\item 
if $x_0^i\in\Omega$ or $x_0\in\partial\Omega$ with $\lim_{m\to\infty} \frac{|x_m^i-x_0^i|}{r_m}=\infty $, then there exists a nonconstant smooth harmonic map $\omega_i:\RR^2\to\SS^2$ with finite energy, such that as $m\to\infty$,
\[ (u_m^i,d_m^i)\to(0,\omega_i) \text{ in } C_{{\rm loc}}^2(\RR^2\times[-\infty,0]);  \]

\item 
if $x_0^i\in\partial\Omega $ and $\lim_{m\to\infty} \frac{|x_m^i-x_0^i|}{r_m}<\infty$, then there exists a nonconstant smooth harmonic map $\omega'_i: \RR_a^2\to\SS^2$ with finite energy for some half plane $\RR_a^2$, such that as $m\to\infty$,
    \[  (u_m^i,d_m^i)\to(0,\omega_i') \text{ in } C_{{\rm loc}}^2(\RR^2_a\times[-\infty,0]), \]
\end{enumerate}
 where
\[ u_m^i(x,t) = r_m^i u(x_m^i+r_m^ix, t_m^i+(r_m^i)^2t), \ d_m^i(x,t) = d(x_m^i+r_m^ix, t_m^i+(r_m^i)^2t). \]

\item Set $T_0=0$. Then, for $0\leq i\leq L-1$,
\[ |d_t| +|\nabla^2d| \in L^2(\Omega\times[T_i,T_{i+1}-\epsilon]),\quad |u_t| + |\nabla^2u|\in L^{\frac43}(\Omega\times[T_i,T_{i+1}-\epsilon]) \]
for any $\epsilon>0$, and for any $0<T_L<T<+\infty$,
\[  |d_t| +|\nabla^2d| \in L^2(\Omega\times[T_L,T]),\quad |u_t| + |\nabla^2u|\in L^{\frac43}(\Omega\times[T_L,T]). \]
\end{enumerate}
\end{theorem}

\begin{remark}
\noindent
\begin{enumerate}
\item Theorem \ref{Thm B} is established by using Theorem \ref{Thm A}, the global and local energy inequalities, and estimates of the pressure function $P$ in Section \ref{sec-3} below. 

\item Compared to \cite{LLW10}, the interior estimates in the proof are similar. However, extra care needs to be taken for the estimates near boundary.

\item The global existence of weak solutions for the nematic liquid crystal flow in three dimensions is a notoriously hard problem, and in the case with Dirichlet boundary, it was solved by Lin and Wang \cite{LW2016CPAM} with the restriction $d_0(\Omega)\subset \mathbb{S}^2_+$. The major difficulties are from the energy supercritical harmonic map heat flow and the three dimensional Navier-Stokes equation with a supercritical forcing term. For the free boundary system \eqref{LCF}-\eqref{PFB}, the situation might be better since the vorticity $\omega=\nabla \times u$ may have better estimates with Navier boundary condition.
\end{enumerate}
\end{remark}

\medskip

When we consider the eternal behavior of solution to \eqref{LCF} with \eqref{PFB} as $t\to+\infty$, it turns out that there are two distinct situations depending on  whether the domain is axisymmetric or not. For non-axisymmetric domain, one can utilize the stronger Korn's inequality in Lemma \ref{lem: Korn} to show that the velocity field finally decays to $u_\infty\equiv 0$. In the case of axisymmetric domain, \eqref{LCF}-\eqref{PFB} in fact permits stationary and circular velocity field. See \cite{{Amrita2018NSthesis}} for related discussions in three dimensions. It remains a problem what properties of initial data $(u_0,d_0)$ can explain the behavior of eternal  weak limit $(u_\infty,d_\infty)$.

 \begin{theorem}[Eternal behavior]\label{Thm C}
 \noindent
 \begin{enumerate}
\item There exist $t_k\nearrow+\infty$, $u_\infty\in H^1(\Omega)$ with $\DD u_\infty = 0$, and a harmonic map $d_\infty \in C^\infty(\Omega,\SS^2)\cap C^{2,\alpha}(\overline{\Omega},\SS^2)$ with $d_\infty$ satisfies $\eqref{PFB}_2$ on $\partial\Omega$ such that $u(\cdot,t_k)\weakto u_\infty$ weakly in $H^1(\Omega)$, $d(\cdot, t_k)\weakto d_\infty$ weakly in $H^1(\Omega)$, and there exist $l,l'\in \NN$, points $\{ x_i\}_{i=1}^l\subseteq\Omega$, $\{y_i\}_{i=1}^{l'}\subseteq\partial\Omega$ and $\{m_i\}_{i=1}^l,\{m_i'\}_{i=1}^{l'}\subseteq\NN$ such that 
\[ |\nabla d(\cdot, t_k)|^2 dx \weakto |\nabla d_\infty|^2 dx + \sum_{i=1}^l 8\pi m_i\delta_{x_i} + \sum_{i=1}^{l'}4\pi m_i' \delta_{y_i} \ \ \text{ in Radon measure.}\]

Moreover, if $\Omega$ is non-axisymmetric, we can further conclude that $u_\infty\equiv 0$.

\item  Suppose that $(u_0,d_0)$ satisfies
\[\int_\Omega |u_0|^2 + |\nabla d_0|^2 \leq 4\pi,\]
then $(u,d)\in C^\infty(\Omega,(0,+\infty))\cap C^{2,1}_\alpha(\overline{\Omega}\times(0,+\infty))$. Moreover, there exist $t_k\nearrow +\infty$, $u_\infty\in H^1(\Omega)$ with $\DD u_\infty=0$, and a harmonic map $d_\infty\in C^\infty(\Omega,\SS^2)\cap C^{2,\alpha}(\overline{\Omega},\SS^2)$ with $d_\infty$ satisfies $\eqref{PFB}_2$ on $\partial\Omega$ such that $(u(\cdot,t_k), d(\cdot,t_k))\to (u_\infty,d_\infty)$ in $C^2(\overline{\Omega})$.

Moreover, if $\Omega$ is non-axisymmetric, we can further conclude that $u_\infty\equiv 0$.
\end{enumerate}
\end{theorem}

Here we give a further classification result of the eternal weak solution $(u_\infty, d_\infty)$ in the situation of axisymmetric domain.

\begin{theorem}[Classification of eternal weak solution in axisymmetric domain] \label{Thm D}
    Suppose that the domain $\Omega$ is axisymmetric. There exist $t_k\nearrow+\infty$ and a harmonic map $d_\infty \in C^\infty(\Omega,\SS^2)\cap C^{2,\alpha}(\overline{\Omega},\SS^2)$ with $d_\infty$ satisfies $\eqref{PFB}_2$ on $\partial\Omega$ such that $(u(\cdot,t_k), d(\cdot, t_k)) \weakto (u_\infty, d_\infty)$ weakly in $H^1(\Omega)$. Moreover, $(u_\infty,d_\infty)$ can be classified as follows.
    \begin{enumerate}
    \item  If $\Omega$ is a disk $B_r$, then
    \begin{equation}
    \begin{cases}
        u_\infty = c(x_2,-x_1),\quad c\neq 0\\
        d_\infty \equiv (K_1,K_2,K_3), \quad K_1^2+ K_2^2 + K_3^2=1
    \end{cases}
    ~or \quad
    \begin{cases}
        u_\infty \equiv 0 \\
        d_\infty \text{ is a harmonic map.}
    \end{cases}
    \end{equation}
    \item  If  $\Omega$ is an annulus  $B_{r_2}\setminus B_{r_1}$, then
    \begin{equation}
    \begin{cases}
        u_\infty = c(x_2,-x_1) \\
        d_1  = K_1 \cos(\alpha \ln(\frac{r}{r_1})) \\
        d_2  = K_2 \cos(\alpha \ln(\frac{r}{r_1}))\\
        d_3  = K_3 \sin(\alpha \ln(\frac{r}{r_1}))
    \end{cases}
    ~or \quad
    \begin{cases}
        u_\infty \equiv 0 \\
        d_\infty \text{ is a harmonic map,}
    \end{cases}
    \end{equation}
    where 
    \begin{equation}\label{eqn: annulus_info}
            |K_3| = 1, \quad  K_1^2 + K_2^2 =1, \quad \alpha\ln\left(\frac{r_2}{r_1}\right) = k\pi, ~k\in\mathbb{Z}, \quad c\in \RR.
    \end{equation}
    \end{enumerate}
\end{theorem}

\begin{remark}
\noindent
\begin{enumerate}
\item We emphasize that such nontrivial weak limit for $u_\infty$ is a new feature triggered by the partially free boundary condition \eqref{PFB}, which is not possible in the case of Dirichlet boundary.
\item The stability of $(u_\infty,d_\infty)$ in this free boundary model might be an interesting and challenging problem.
\end{enumerate}
\end{remark}

\medskip

The rest of this paper is devoted to the proof of Theorem \ref{Thm A}, Theorem \ref{Thm B}, Theorem \ref{Thm C} and Theorem \ref{Thm D}.

\bigskip

\section{Notations and preliminaries}

\medskip

In this section, we introduce some notations and estimates that will be used throughout this paper.

We use $A\lesssim B$  to denote $A\leq CB$ for some universal constant $C>0$. For $x_0\in\RR^2, t_0\in\RR$, $z_0=(x_0,t_0)$, denote 
\begin{align*}
    B_r(x_0) & = \{ x\in\RR^2 : |x-x_0|\leq r \},  \\
    P_r(z_0) &  = B_r(x_0) \times [t_0-r^2,t_0]
\end{align*}
to be spatial neighborhood and parabolic cylinder, respectively. We use $\Omega_t$ to denote $\Omega\times[0,t]$.
For $ x_0\in\partial\Omega$, we use $Q_r(z_0)^+$ to denote $P_r(z_0)\cap \Omega_t$ to denote the parabolic cylinder at boundary.
Denote the boundary of parabolic cylinder $\partial_pP_r(z_0)$ to be 
\[ \partial_p P_r(z_0) = (B_r(x_0)\times\{t_0-r^2\})\cup(\partial B_r(x_0)\times [t_0-r^2,t_0]) \]

For $1<p,q<\infty$, denote $L^{p,q}(P_r(z_0))= L^q([t_0-r^2,t_0], L^p(B_r(x_0))$ with norm
\[ \|f\|_{L^{p,q}(P_r(z_0))} = \Big( \int_{t_0-r^2}^{t_0} \|f(\cdot,t)\|_{L^p(B_r(x_0))}^q dt \Big)^{\frac1q} \]
Further, denote $W^{1,0}_{p,q}(P_r(x_0)) = L^q([t_0-r^2,t_0], W^{1,p}(B_r(x_0))) $, with norm 
\[ \|f\|_{W^{1,0}_{p,q}(P_r(x_0))} = \|f\|_{L^{p,q}(P_r(z_0))} + \|\nabla f\|_{L^{p,q}(P_r(z_0))} \]
Denote $W^{2,1}_{p,q}(P_r(x_0)) = \{f\in W^{1,0}_{p,q}(P_r(x_0)): \nabla^2 f , \partial_t f \in L^{p,q}(P_r(x_0))\}$, with norm
\[ \|f\|_{W^{2,1}_{p,q}(P_r(x_0))} = \|f\|_{W^{1,0}_{p,q}(P_r(x_0))} + \|\nabla^2f\|_{W^{1,0}_{p,q}(P_r(x_0))} + \|\partial_t f\|_{W^{1,0}_{p,q}(P_r(x_0))} \]
If $p=q$, then the above notation can be simplified as $L^{p,p}=L^p, W^{1,0}_{p,p}=W^{1,0}_p, W^{2,1}_{p,p}=W^{2,1}_p$.

Here are some techniques that we are going to utilize:

\begin{lemma}[Gradient estimates for the heat equation]\label{heat: gradient} If $d: \Omega\to\RR^d$ solves heat equation  $ \partial_td-\Delta d=0$ with Dirichlet boundary condition ($d=0$ on $\partial\Omega$) or Neumann boundary condition ($\n\cdot\nabla d=0$ on $\partial\Omega$), then we have the following integral estimates for the gradient $\nabla d$: 
\begin{equation*} \begin{split} 
 \text{interior:\  } &  \int_{P_{\theta R}(z_1)} |\nabla d|^4\lesssim \theta^4 \int_{P_{R}(z_1)} |\nabla d|^4, \\
\text{boundary:\  } & \int_{P^+_{\theta R}(z_1)} |\nabla d|^4 \lesssim \theta^4 \int_{P^+_R(z_1)} |\nabla d|^4.
\end{split}
\end{equation*}
\end{lemma}
\begin{proof}
    See  \cite[Lemma 4.5, Lemma 4.13, Lemma 4.20, Theorem 7.35]{Lieberman96}.  
\end{proof}

\begin{lemma}[Parabolic Morrey's decay lemma]\label{parabolic morrey}
    Suppose for any $z\in \overline{\Omega}$ and any $0<r<\min(\diam(\Omega),\sqrt{T})$ we have 
    \[  \int_{P_r(z)\cap\Omega_T} |\nabla d|^p + r^p|d_t|^p \lesssim r^{n+2+(\alpha-1)p} , \]
    then $d\in C^\alpha(P_{\frac12}(z_1)\cap\Omega_T)$.
\end{lemma}
\begin{proof}
We use the definition of Campanato space and apply Poincar\'e inequality \[ r^{-(n+2+\alpha p)}\int_{P_r(z)\cap\Omega_T} |d-d_{z,r}|^p \lesssim r^{-(n+2+ \alpha p)}\int_{P_r(z)\cap\Omega_T} r^p|\nabla d|^p + r^{2p}  |d_t|^p \lesssim 1.\]\end{proof}  

\begin{lemma}[A variant of Ladyzhenskaya's inequality]\label{Ladyzhenskaya} There exists $C_0$ and $R_0$ depending only on $\Omega$ such that for any $T>0$, if $u\in L^{2,\infty}(\Omega_T)\cap W^{1,0}_2(\Omega_T)$, then for $R\in(0,R_0)$,
\[ \int_{\Omega_T} |u|^4 \leq C_0 \sup_{(x,t)\in \overline{\Omega}_T}\int_{\Omega\cap B_R(x)} |u|^2(\cdot,t) \Big(\int_{\Omega_T} |\nabla u|^2  + \frac{1}{R^2}\int_{\Omega_T} |u|^2 \Big).\]
\end{lemma}
\begin{proof}
    See \cite[Lemma 3.1]{Struwe85}.
\end{proof}

\begin{lemma}[Refined embedding theorem]\label{refined embedding} For $u\in W^{1,p}_0(\Omega)$ with $1\leq p< n$ and $1\leq r\leq p^*=\frac{np}{n-p}$, we have for any $q\in[r, p^*]$,
\[ \|u\|_{L^q(\Omega)} \lesssim \|\nabla u \|^\alpha_{L^p(\Omega)} \|u\|^{1-\alpha}_{L^r(\Omega)},  \]
where $\alpha = (\frac1r-\frac1q)(\frac1r-\frac1{p^*})^{-1}$.    
\end{lemma}

\begin{lemma}[$L^p-L^q$ regularity for Neumann heat equation] \label{lem: LpLq reg Heat} Let $1\leq p\leq q\leq \infty, (q\neq 1, p\neq \infty)$.
    Let $ (e^{t\Delta})_{t\geq 0} $ be the Neumann heat semigroup in $\Omega$, and let $\lambda_1>$  0 denote the first nonzero
eigenvalue of $-\Delta$ in $\Omega$ under Neumann boundary conditions. Then there exist constant $C(\Omega,t_0)$ such that
\begin{equation*}
\begin{split}    
\|e^{t\Delta} f \|_{L^q(\Omega)} &\leq C(\Omega,t_0) t^{-\frac{n}{2}(\frac{1}{p} - \frac{1}{q})} e^{-\lambda_1 t} \| f \|_{L^p(\Omega)}  \\
\| \nabla e^{t\Delta} f \|_{L^q(\Omega)} & \leq C(\Omega,t_0) t^{-\frac12-\frac{n}{2}(\frac1p-\frac1q)} e^{-\lambda_1 t} \|f\|_{L^p(\Omega)}
\end{split}
\end{equation*}
for $t\leq t_0$.
\end{lemma}
\begin{proof}
    See \cite[Lemma 1.3]{Winkler10}.
\end{proof}

\begin{lemma}[$L^p-L^q$ regularity for Stokes operator with Navier boundary condition]\label{lem: LpLq reg Stokes} 
Let  $\Omega\subseteq\RR^n$, $n\geq 2$ be a bounded $C^3$-smooth domain. Let $1 < p \leq q < \infty$, $\delta>0$, let $\PP_p: L^p(\Omega) \to L^p_\sigma(\Omega)$   denote the Helmholtz projection, and let $\bbA_p=\PP_p \Delta$ be the Stokes operator with Navier boundary condition.
Suppose either one of the following two conditions holds:
\begin{enumerate}
    \item If $p<\frac{n}{2}$ and $p\leq q \leq \frac{np}{n-2p}$ where $0\leq \gamma = \frac{n}{2}(\frac{1}{p} - \frac{1}{q})\leq 1$
    \item  If $p\geq\frac{n}{2}$ and $p\leq q$ where $1\geq \gamma\geq 1-\frac{p}{q} \geq 0$,
\end{enumerate}
then
\[ \|e^{-t\bbA_p} u \|_{L^q(\Omega)} \leq C(\delta,p,\Omega) \Big( \frac{t+1}{t}\Big)^\gamma e^{\delta t} \|u\|_{L^p(\Omega)} \]
for $u\in L^p_\sigma(\Omega)$.
\end{lemma}
\begin{proof}
    See  \cite[Corollary 1.4]{FR16}. Based on this inequality, one can also utilize interpolation of Sobolev space to derive
    \[ \|e^{-t\bbA_p} \nabla u \|_{L^q(\Omega)} \leq C(\delta,p,\Omega, t_0) t^{-\frac12-\gamma}  \|u\|_{L^p(\Omega)} \]
    for $ t\leq t_0$.
\end{proof}
\begin{remark}
\noindent
\begin{itemize}
\item    In the setup of Stokes operator with zero Dirichlet boundary condition, the $L^p-L^q$ regularity is in the form of Lemma \ref{lem: LpLq reg Heat}. This is proved by showing that the resolvent operator $R(\lambda,\bbA_p)=(\lambda I + \bbA_p)^{-1}$ is sectorial, and then study the fractional and purely imaginary power of $\bbA_p$, finally conclude the $L^p-L^q$ regularity by using Komatsu semigroup decaying inequality (cf. \cite[Theorem 12.1]{Komatsu66}). In dimension three, such scheme also works in the setup of Stokes operator with Navier-type boundary condition (i.e. $u \cdot\n = \curl u \times \n = 0$, see \cite{ABAE17}), as well as in the setup Navier slip boundary condition (i.e. $2[(\DD u)n]_\tau +\alpha u\cdot \tau = 0 $, see \cite{AEG18}). In the latter situation, either a nontrivial friction $\alpha$ is required, or the non-axisymmetric property of the domain $\Omega$ is required. This suggests that in the current setup of two dimensional Navier perfect-slip boundary condition, generally we  have no $\calH^\infty$-calculus (consider the counterexample of stationary vortex flow in disk). Indeed,  Lemma \ref{lem: LpLq reg Stokes} has no long-time decaying property and its proof in \cite{FR16} is based on Equivalent norms on $D(\bbA_q)$  and interpolation of Sobolev space, without using fractional semigroup.
\item See also \cite[Theorem 3.10]{KLW-PKSNS} for related semigroup estimates.
\end{itemize}
\end{remark}

\begin{lemma}[Parabolic Sobolev embedding theorem]\label{Sobolev embedding} We have continuous embedding $W^{2l,l}_q(Q_T) \subseteq W^{s,r}_p(Q_T) $ if $2l-2r-s-(\frac1q-\frac1p)(n+2)\geq 0$.
\end{lemma}
\begin{proof}
    See \cite[Lemma 3.3]{LSU68}.
\end{proof}

\begin{lemma}[Boundary $W^{2,1}_{p,q}$-estimate for Stokes equation]\label{lem: W21pq} For a homogeneous, non-stationary Stokes equation $\partial_t u - \Delta u  + \nabla P= 0$ in $\Omega$ with Navier boundary condition on $\partial\Omega$. For $s\leq p$ and arbitrary $p_0(t)$, we have
\[ \|u\|_{W^{2,1}_{p,q}(Q^+_{\frac12})}  + \|\nabla P\|_{L^{p,q}(Q^+_{\frac12})} \lesssim \|\nabla u\|_{L^{s,q}(Q_1^+)} + \|P-p_0\|_{L^{s,q}(Q_1^+)}.\]
\end{lemma}
\begin{proof} 
    For a homogeneous, non-stationary Stokes equation with zero Dirichlet boundary condition, this boundary $W^{2,1}_{p,q}$-estimate  has been proved  in   \cite[Lemma 3.2]{SSS04}, using the property of maximal $L^p-L^q$ regularity for Stokes system with nonzero divergence. Such property is intensively studied by  Shibata-Shimizu in \cite{SS12}, \cite{SS08}, and \cite{SS05}. For maximal $L^p-L^q$ regularity in the setting of compressible fluid, see \cite{Kakizawa11}. However, in the setting of Stokes system with Navier boundary condition and nonzero divergence, such a maximal regularity is false because there is no control of the Neumann term as in \cite{SS08}. Instead, we follow Seregin's argument   in \cite{Seregin00}, and we use maximal $L^p-L^q$ regularity for Stokes system with Navier boundary and zero divergence in \cite{FR19}.

    Let $v_1(x,t)=\phi(x) u(x,t)$ and $P_1=\phi(x) (P(x,t)-p_0(t))$ for some cutoff function $\phi\equiv1$ on $B_{\frac12}^+$, $\supp \phi\subseteq B_1^+$, and $\n\cdot\nabla \phi=0$ on $\partial\Omega$.    Then $(v_1,P_1)$ solves 
    \[ \begin{cases}
        \partial_t v_1 -\Delta v_1 + \nabla P_1   = g_1 &  \text{ in } \Omega\\
        \nabla\cdot v_1  = u\cdot \nabla \phi & \text{ in } \Omega \\ 
        v_1 \text{ satisfies Navier slip boundary condition } \eqref{PFB}_1 & \text{ on } \partial\Omega,
    \end{cases}\]
    where $g_1=-2\nabla u\nabla\phi - u\Delta \phi + (P-p_0)\nabla\phi$.

    Further, we let $(v_2(x),P_2(x))$ solves the following stationary Stokes system
    \[ \begin{cases}
        - \Delta v_2 + \nabla P_2 = 0 & \text{ in } \Omega \\
        \nabla \cdot v_2 = u\cdot\nabla\phi  & \text{ in } \Omega \\
        v_2 \text{ satisfies } \eqref{PFB}_1 & \text{ on } \partial\Omega,
    \end{cases} \]
    and $(\partial_t v_2,\partial_t P_2)$ solves
    \[ \begin{cases}
        -\Delta (\partial_t v_2) + \nabla (\partial_t P_2) = 0 & \text{ in } \Omega \\
        \nabla\cdot(\partial_t v_2) = \partial_tu\cdot\nabla \phi & \text{ in } \Omega \\
        \partial_t v_2 \text{ satisfies } \eqref{PFB}_1 & \text{ on } \partial\Omega.
    \end{cases}\]
    Thus $v_3=v_1-v_2$ and $P_3=P_1-P_2$ solve
    \[ \begin{cases}
        \partial_t v_3 -\Delta v_3 + \nabla P_3 = g_3 = g_1-\partial_tv_2  & \text{ in } \Omega \\
        \nabla\cdot v_3 = 0 & \text{ in } \Omega \\
        v_3 \text{ satisfies } \eqref{PFB}_1 & \text{ on } \partial\Omega.
    \end{cases} \]
Let $\chi(t):\RR\to[0,1]$ be the smooth function such that $\chi\equiv0$ on $(-\infty, -1)$ and $\chi\equiv 1$ on $(-\frac14,\infty)$. We see that $v_4=v_3\chi$ and $P_4=P_3\chi$ solve
    \[ \begin{cases}
        \partial_t v_4 -\Delta v_4 + \nabla P_4 = g_4 = \chi g_3 + v_3\partial_t\chi & \text{ in } \Omega \\
        \nabla \cdot v_4 = 0 & \text{ in } \Omega \\
        v=0 & \text{ at } t=-1 \\
        v \text{ satisfies } \eqref{PFB}_1 & \text{ on } \partial\Omega.
    \end{cases} \]
    By the maximal $L^p-L^q$ regularity to the Stokes system with Navier perfect slip boundary condition (see \cite[Theorem 2.9]{FR19}), we have
    \[ \|v_4\|_{W^{2,1}_{p,q}(\Omega\times(-1,0))} + \|\nabla P_4\|_{L^{p,q}(\Omega\times(-1,0))} \lesssim \|g_4\|_{L^{p,q}(\Omega\times(-1,0))}.  \]
First we can directly estimate $g_1$ that
    \[ \|g_1(\cdot,t)\|_{L^p(\Omega)} \lesssim \|  u\|_{W^1_p(B_1^+)}  + \|P-p_0\|_{L^p(B_1^+)}. \]
To estimate $\partial_tv_2$ in $g_3$, we notice that $\partial_tu\cdot\nabla\phi = \nabla a +b$ where $a=-(P-p_0)\nabla\phi  + \nabla u\nabla \phi$ and $b=-\nabla u:\nabla^2\phi  + (P-p_0)\Delta\phi$, we can check that the duality argument of  \cite[Theorem 2.4]{Solonnikov77} also works with Navier boundary condition, and it gives
    \[ \|\partial_tv_2\|_{L^p(\Omega)} \lesssim \|a\|_{L^p(\Omega)} + \|b\|_{L^p(\Omega)} + \|a\cdot\n\|_{L^p(\partial\Omega)}, \]
    while the last term can be estimated by 
    \[\|a\cdot\n\|_{L^p(\partial\Omega)} \lesssim \delta^{\frac1q}\big( \|\nabla^2u\|_{L^p(B_1^+)}  + \|\nabla P\|_{L^p(B_1^+)}  \big) + \Big(\frac1\delta\Big)^{\frac{p'}{pq}}\big( \|u\|_{W^1_p(B_1^+)}  + \|P-p_0\|_{L^p(B_1^+)} \big). \]
We can also use test function $\psi\in C^\infty(\Omega)$ with $\n\cdot\nabla\psi=0$ on $\partial\Omega$ to estimate that 
    \[ \|v_2\|_{W^2_p(\Omega)}   + \|\nabla P_2\|_{L^p(\Omega)} \lesssim \|\nabla(u\cdot\nabla\phi)\|_{L^p(\Omega)} \lesssim \|u\|_{W^{1,p}(B_1^+)}.\]
Combining the above estimates, it follows that 
    \[ \|u\|_{W^{2,1}_{p,q}(Q_{\frac12}^+)} + \|\nabla P\|_{L^{p,q}(Q_{\frac12}^+)} \lesssim \|\nabla u\|_{L^{p,q}(Q_1^+)} + \| P - p_0\|_{L^{p,q}(Q_1^+)}. \]
Moreover, for $s\leq p$, one can use Sobolev embedding to lift above RHS to terms with $W^{2,1}_{s,q}$ norm and $W^{1,0}_{s,q}$ norm, then apply the above inequality again to deduce that
        \begin{equation}\label{eqn:  W21pqr} \|u \|_{W_{p,q}^{2,1}(Q_{\frac12}^+)} + \|\nabla P \|_{L^{p,q}(Q_{\frac12}^+)}  \lesssim \|\nabla u\|_{L^{s,q}(Q_1^+ )} + \|P-p_0\|_{L^{s,q}(Q_1^+)} . \end{equation}
\end{proof}

\begin{lemma}[Riesz potential estimates between parabolic Morrey space]\label{lem: Riesz-Morrey} Let $\widetilde{M}^{p,\lambda}(\Omega_T)$ and $\widetilde{M}^{p,\lambda}(\Omega_T)_*$, respectively denote the parabolic Morrey space, and the weak parabolic Morrey space, where $1 \leq p <\infty$, $0 \leq\lambda\leq n + 2$, and 
\begin{align*} 
\|f\|_{\widetilde{M}^{p,\lambda}(\Omega_T)}^p & := \sup_{r>0, z\in \Omega_T} r^{\lambda-(n+2)} \| f\|^p_{L^p(\Omega_T\cap P_r(z))},\\
\|f\|_{\widetilde{M}_*^{p,\lambda}(\Omega_T)}^p  & := \sup_{r>0, z\in \Omega_T} r^{\lambda-(n+2)} \| f\|^p_{L^{p,*}(\Omega_T\cap P_r(z))}.
\end{align*}
Let $\widetilde{I}_\beta (f)$ denote the parabolic Riesz potential of order $\beta\in[0,n+2]$,    \[ \widetilde{I}_\beta(f)=\int_{\RR^{n+1}} \frac{f(y,s)}{\delta((x,t),(y,s))^{n+2-\beta}} dyds \]
where $f\in L^p(\RR^{n+1})$ and parabolic distance $\delta((x,t),(y,s)) = \max(|x-y|, \sqrt{|t-s|}) $.
Then for any $\beta>0, 0<\lambda\leq n+2, 1<p<\frac{\lambda}{\beta}$, if $f\in L^p(\RR^{n+1})\cap\widetilde{M}^{p,\lambda}(\RR^{n+1})$, then $\widetilde{I}_\beta (f) \in L^{\tilde p}(\RR^{n+1})\cap \widetilde{M}^{\tilde p,\lambda}(\RR^{n+1})$ with $\tilde p = \frac{p\lambda}{\lambda-p\beta}$. Further, for any $0<\beta<\lambda\leq n+2$, if $f\in L^1(\RR^{n+1})\cap \widetilde{M}^{1,\lambda}(\RR^{n+1})$, then $\widetilde{I}_\beta(f)\in L^{ \frac{\lambda}{\lambda-\beta}, *}(\RR^{n+1})\cap \widetilde{M}^{\frac{\lambda}{\lambda-\beta}, \lambda}_*(\RR^{n+1})$.
\end{lemma}
\begin{proof}
    See \cite[Theorem 3.1]{HW10}.
\end{proof}

\begin{lemma}[Density of smooth maps in $\Lbf$ and $\Jbf$]\label{lem: density_Sobolev}
    For $n=2$, and any given map $v\in \Lbf $ and $f\in \Jbf$, there exists sequence $\{v_k\}\subseteq C^\infty(\Omega,\SS^2)\cap C^{2,\alpha}(\overline{\Omega},\SS^2)\cap\Lbf$ and $\{f_k\}\subseteq C^\infty(\Omega,\RR^2)\cap C^{2,\alpha}(\overline{\Omega},\RR^2)\cap \Jbf$ such that  
    \[  \lim_{k\to\inf} \|v_k-v\|_{L^2(\Omega)} = \lim_{k\to\infty} \|f_k-f\|_{H^1(\Omega)} = 0 .\]
\end{lemma}
\begin{proof}
    For classical result that smooth maps are dense in Sobolev maps, see \cite{Bethuel91}. Since we need lots of other properties such as H\"older continuity, Navier boundary condition for $v_k$, and free boundary condition for $f_k$, our method is to consider solution of evolution equation with initial data $v$ and $f$ with  free boundary condition $\eqref{PFB}_{1}$, then we take backward slices $v_k = v(\cdot,t=2^{-k})$ and $f_k=f(\cdot,2^{-k})$ as smooth approximations. 
    For director field $f(x,t):\Omega\times(0,T)\to\SS^2$, we consider the heat flow of harmonic map
    \[ \begin{cases}
    f_t = \Delta f + |\nabla f|^2f & \text{in } \Omega\times(0,t_f) \\
    f(\cdot,0) = f_0 \in \Jbf \\
    \nabla_{\n} f_1 = \nabla_{\n} f_2 = f_3 = 0 & \text{on } \partial\Omega\times(0,t_f). \\
    \end{cases}\]
    Classical result in short time existence and regularity shows that if $t_f$ is smaller than the first singular time, then $f(x,t)\in C^{\infty}(\Omega\times(0,t_f))$, see \cite{Struwe88}. That argument is the same as what we shall use in proving Theorem \ref{Thm B}. Smoothness of solution of equation gives $f\in C([0,T], H^1(\Omega))$ and thus $\lim_{k\to\infty}\| f_k - f\|_{H^1(\Omega)}  =0$. H\"older regularity follows easily from estimation of Morrey-Campanato norm, and the argument is the same as what we shall do in proving Theorem \ref{Thm A}. 
    
    Similarly, consider solution of Navier-Stokes equation with initial data $v_0$ with Navier-slip boundary condition $\eqref{PFB}_2$.
    \[ \begin{cases}  
    v_t + v\cdot\nabla v - \Delta v + \nabla P = 0 & \text{in } \Omega\times(0,t_v) \\
    \nabla\cdot v = 0 & \text{in } \Omega\times(0,t_v) \\
    v(\cdot,0) = v_0 \\
    (\DD u\cdot\n)_{\tau} = u\cdot\n = 0 & \text{on } \partial\Omega\times(0,t_v).
    \end{cases} \]
    We similarly use existence and regularity for 2D Navier-Stokes to show that  $v\in C([0,T],L^2(\Omega))$ and thus $\lim_{k\to\infty} \|v_k-v\|_{L^2(\Omega)}=0$, see \cite{Kelliher06}. Also, H\"older regularity follows from estimation of Morrey-Campanato norm which is the   argument we shall use in Theorem \ref{Thm A}. Alternatively, since the proof of Theorem \ref{Thm A} does not require this lemma, we can also use the result of Theorem \ref{Thm A} with $u\equiv 0$ or $d\equiv C$ to obtain the H\"older regularity.
\end{proof}

\begin{lemma}[Korn's inequality]\label{lem: Korn} Let $\Omega\subseteq\RR^n$ with $n\geq 2$ be an open,  bounded, and $C^1$-smooth domain. Let $v$ be a vector field on $\Omega$ with $\nabla v\in L^2(\Omega)$. Then there exists $C(\Omega)>0$ such that
\[ \|v\|_{L^2(\Omega)} + \|\DD v\|_{L^2(\Omega)} \geq C(\Omega) \|\nabla v\|_{L^2(\Omega)}.  \]

Assume further that $\Omega$ is non-axisymmetric and $v$  has tangential boundary condition (i.e. $v\cdot\n =0$ on $\partial\Omega$), then
\[ \|\DD v\|_{L^2(\Omega)} \geq K(\Omega) \|\nabla v\|_{L^2(\Omega)} \]
for some constant $K(\Omega)>0$ measuring the deviation of $\Omega$ from being axisymmetric. 
\end{lemma}
\begin{proof}
    See \cite[Theorem 3]{DV02}.
\end{proof}

\bigskip


\section{H\"older regularity of solution}\label{sec-3}

\medskip

In this section, we use the standard hole-filling argument to obtain regularity of solution $v$ to a parabolic system $(P)$. This argument has mainly two parts: first, we cut a hole (e.g. $P_r(x_0)$ or $\Omega_T$) and fill it with the solution $v'$ of homogeneous system $(P')$ (e.g. heat equation, Stokes system, harmonic map heat flow) with boundary condition inherited from the original solution $v$. Then $v-v'$ solves the parabolic system $(P-P')$, which is usually  inhomogeneous but has a good boundary condition on the hole. This gives us more freedom to perform integration by parts and allows us to handle  inhomogeneity only. We leave boundary condition to good system system $(P')$ and we deal with it independently. In the end, we achieve a gradient estimate and obtain $C^\alpha$-regularity of $v$. The next step is to use $C^\alpha$-regularity to obtain $C^{1,\beta}$-regularity for some $\beta\in(0,1)$. One approach given by \cite{HW10} is to analyze the Riesz potential between parabolic Morrey space. We apply the first part of hole-filling argument in the following three lemmas, and apply the second part in the proof of Theorem \ref{Thm A}.

We start with the $L^{\frac43}$ estimate of $|\nabla P|$ to the solution $(u,d,P)$ of the free boundary system \eqref{LCF}-\eqref{assumption1}.
\begin{lemma}[Estimation of pressure]\label{lem: est pressure} For $0<T<\infty$, suppose that $u\in L^{2,\infty}\cap W^{1,0}_2(\Omega_T)$ and $d\in L^{\infty}([0,T], H^1(\Omega))\cap L^2([0,T], H^2(\Omega))$ is a weak solution to \eqref{LCF}-\eqref{assumption1}. Then $\nabla P\in L^{\frac43}(\Omega_T)$ and we have the following estimate
\[ \|P-P_\Omega\|_{L^{4,\frac43}(\Omega_T)}\lesssim \|\nabla P\|_{L^{\frac43}(\Omega_T)}\lesssim \|u\|_{L^4(\Omega_T)}\|\nabla u\|_{L^2(\Omega_T)} + \|\nabla d\|_{L^4(\Omega_T)}\|\nabla^2d\|_{L^2(\Omega_T)}. \]    
\end{lemma}
\begin{proof}
    We write $u=v^1+v^2$ where $v^1:\Omega_T\to\RR^2$ solves homogeneous heat equation $\partial_t v^1 -\Delta v^1=0$ with initial condition $v^1=u_0$ on $\Omega\times\{t=0\}$ and partially free boundary condition ($v^1$ satisfies $\eqref{PFB}_1$ on $\partial\Omega$). Then $v^2=v-v^1$ solves non-homogeneous, non-stationary Stokes equation $\partial_t v^2 - \Delta v^2 + \nabla P = - u\cdot u - \nabla\cdot(\nabla d \odot \nabla d - \frac12|\nabla d|^2\Id_2)$ with initial condition $v^2=0$ on $\Omega\times\{t=0\}$ and partially free boundary condition $v^2$ satisfies $\eqref{PFB}_1$ on $\partial\Omega$). Then Sobolev inequality and $W^{2,1}_{\frac43}$-estimate give
    \begin{align*} 
     \|P-P_\Omega\|_{L^{4,\frac43}(\Omega_T)} & \lesssim \|\nabla P\|_{L^{\frac43}(\Omega_T)} \lesssim \||u||\nabla u|\|_{L^{\frac43}(\Omega_T)} + \||\nabla d||\nabla^2d|\|_{L^{\frac43}(\Omega_T)}  \\
    &\lesssim \|u\|_{L^4(\Omega_T)}\|\nabla u\|_{L^2(\Omega_T)} + \|\nabla d\|_{L^4(\Omega_T)}\|\nabla^2d\|_{L^2(\Omega_T)}. \end{align*}
\end{proof}

\begin{lemma}[Local smallness]\label{lem:local_smallness} For any $\alpha\in(0,1)$, there exists $\epsilon_0>0$ such that for $z_0=(x_0,t_0)\in\RR^3$ and $r>0$, if $(u,d)\in W_2^{1,0}(P_{2r}(z_0)), P\in W^{1,\frac43}(P_{2r}(z_0))$ is a weak solution to \eqref{LCF}  and 
\begin{equation}\label{smallness_local}
    \int_{P_r(z_0)} |u|^4+|\nabla d|^4 \leq\epsilon_0^4,
\end{equation}
then $(u,d)\in C^\alpha(P_{\frac r2}(z_0), \RR^2\times\SS^2)$. Moreover,
\begin{align}
    \label{local_holder_d} [d]_{C^\alpha(P_{\frac r2} (z_0))} & \leq C(\|u\|_{L^4(P_r(z_0))} + \|\nabla d\|_{L^4(P_r(z_0))} ), \\
    \label{local_holder_u} [u]_{C^{\alpha}(P_{\frac r2} (z_0))} & \leq C ( \|u\|_{L^4(P_r(z_0))} + \|\nabla d\|_{L^4(P_r(z_0))} + \|\nabla P\|_{\frac43(P_r(z_0))} ).
\end{align}
\end{lemma}
\begin{proof}
The proof is similar to that of \cite[Lemma 2.2]{LLW10}. Here we give a sketch for self-containedness. \\
    
    \emph{Step 1:} First we want to have a local growth control of $|\nabla d|$ in the following sense:
    \begin{equation}\label{est: interior growth d} \int_{P_r(z_1)} |\nabla d|^4 \leq \Big(\frac{r}R\Big)^{4\alpha} \int_{P_R(z_1)} |\nabla d|^{4 } \text{ for } 0<r\leq R.\end{equation}
    To do so, we decompose $d=d^1+d^2$ where $d_1: P_R(z_1)\to\RR^3$ satisfies the homogeneous heat equation $\partial_t d^1-\Delta d^1=0$ with Dirichlet boundary condition $d^1=d$ on $\partial P_R(z_1)$, and $d^2: P_R(z_1)\to\RR^3$ satisfies the inhomogeneous heat equation $\partial_t d^2 - \Delta d^2 = -u\cdot d + |\nabla d|^2 d$ with zero boundary condition $d^2=0$ on $\partial P_R(z_1)$.  We test $d^2$ equation with $\Delta d^2$ and apply  H\"older inequality to obtain
    \[ \sup_{t_1-R^2\leq t\leq t_1} \int_{B_R(x_1)} |\nabla d^2|^2 + \int_{P_R(z_1)} |\Delta d^2|^2 \lesssim (\|u\|^2_{L^4(P_R(z_1))} + \|\nabla d\|^2_{L^4(P_R(z_1))} )\|\nabla d\|^2_{L^4(P_R(z_1))}. \]
    This, together with Ladyzhenskaya inequality in Lemma \ref{Ladyzhenskaya}, yields
    \[ \int_{P_R(z_1)} |\nabla d^2|^4 \lesssim \Big( \int_{P_R(z_1)} |u|^4 + \int_{P_R(z_1)} |\nabla d|^4 \Big) \int_{P_R(z_1)} |\nabla d|^4 .\] We use Lemma \ref{heat: gradient} for $d^1$   and obtain $\|\nabla d^1\|_{L^4(P_{\theta R})} \lesssim \theta \| \nabla d^1\|_{L^4(P_R)}$, and thus
    \[ \int_{P_{\theta R}(z_1)} |\nabla d^1|^4 \leq C \theta^4 \int_{P_R(z_1)} |\nabla d|^4 + C \Big( \int_{P_R(z_1)} |u|^4 + \int_{P_R(z_1)} |\nabla d|^4 \Big) \int_{P_R(z_1)} |\nabla d|^4, \] which gives inequality \eqref{est: interior growth d} if we choose $\theta=r/R<\theta_0$ and $\epsilon_0<\theta_0$ such that $2C\theta_0^4\leq \theta_0^{4\alpha}$.
   
     \emph{Step 2:} Next we want to control the local growth of $|d_t|$. We test $\eqref{LCF}_3$ with  $d_t \phi^2$ with cut-off function $\phi\in C_0^\infty(B_r(x_0)), 0\leq \phi\leq 1, \phi\equiv 1$ on $B_{\frac{r}2}(x_0)$, and $|\nabla \phi|\leq 2/r$. It follows that
    \[ \int_{B_{\frac{r}2}(x_0)} |d_t|^2\phi^2 + \frac{d}{dt}\int_{B_r(x_0)} |\nabla \phi|^2\phi^2 \lesssim r^{-2}\int_{B_r(x_0)}|\nabla d|^2 + \int_{B_r(x_0)} |u|^2|\nabla d|^2.  \]
    We select $s_0\in(t_0-r^2-t_0-\frac{r^2}{4})$ such that $\int_{B_r(x_0)}|\nabla(\cdot,s_0)|^2\lesssim r^{-2}\int_{P_r(z_0)}|\nabla d|^2$.
     As a consequence, 
     \begin{equation}\label{est: interior growth dt} \int_{P_{\frac{r}2}(z_0)} |d_t|^2 
        \lesssim (1+\|u\|^2_{L^4(P_r(z_0))})\|\nabla d\|^2_{L^4(P_r(z_0))}
     \lesssim \Big( \frac{r}{R} \Big)^{2\alpha}\Big( \int_{P_R(z_0)} |\nabla d|^4 \Big)^{\frac12}. \end{equation}
     Estimations \eqref{est: interior growth d} and \eqref{est: interior growth dt} imply the H\"older continuity by Lemma \ref{parabolic morrey}. 

    \emph{Step 3:} We will estimate the local growth of $|\nabla u|$ by separating  solution $u=u^1+u^2$ of $\eqref{LCF}_1$ into two functions, where $u^1:P_R(z_1)\to\RR^2$ satisfies homogeneous heat equation $\partial_t u^1 -\Delta u^1=0$  with Dirichlet boundary condition $u^1=u$ on $\partial P_R(z_1)$ and $u^2=u-u^1:P_R(z_1)\to\RR^2$ satisfies inhomogeneous heat equation $\partial_t u^2 - \Delta u^2 + \nabla P = -\nabla\cdot[u\otimes (u-u_{z_1,R}) + \nabla d\odot \nabla d -\frac12 |\nabla d|^2 \Id_2 ]$ with zero boundary condition $u^2=0$ on $\partial P_R(z_1)$. We test it with $u^2$ and use  H\"older inequality to estimate
    \[ \frac12\frac{d}{dt}\int_{B_R(x_1)} |u^2|^2 + \int_{B_R(x_1)} |\nabla u^2|^2 \lesssim \int_{B_R(x_1)} |\nabla d|^4 + |u|^2|u-u_{z_1,R}|^2 + |u^2||\nabla P| .\]
    We integrate over $[t_1-R^2,t_1]$ and use H\"older inequality to obtain
    \[ \sup_{t_1-R^2\leq t\leq t_1} \int_{B_R(x_1)} |u^2|^2 + \int_{P_R} |\nabla u^2|^2 \lesssim \int_{P_R}|\nabla d|^4  + \|u\|^2_{L^4(P_R)}\|u-u_{z_1,R}\|^2_{L^4} + \|u^2\|_{L^4(P_R)}\|\nabla P\|_{L^{\frac43}(P_R)}, \]
    which, together with Ladyzhenskaya inequality, yields
        \[  \int_{P_R(z_1)}| u^2|^4  \lesssim \Big( \int_{P_R(z_1)} |\nabla d|^4 \Big)^2 + \int_{P_R(z_1)}|u-u_{z_1,R}|^4\cdot\int_{P_R(z_1)}|u|^4 + \Big( \int_{P_R(z_1)} |\nabla P|^{\frac43} \Big)^3. \]
    Collecting the above two estimates gives
    \[    \Big( \int_{P_R(z_1)}|\nabla u^2|^2 \Big)^2 \lesssim \Big( \int_{P_R(z_1)} |\nabla d|^4 \Big)^2 + \int_{P_R(z_1)}|u-u_{z_1,R}|^4\cdot\int_{P_R(z_1)}|u|^4 + \Big( \int_{P_R(z_1)} |\nabla P|^{\frac43} \Big)^3. \] 

    Applying Lemma \ref{heat: gradient} to $u^1$ we have $\|u^1-u^1_{z_1,\theta R}\|_{L^4(P_{\theta R}(z_1))} \lesssim \theta^2 \|u^1-u^1_{z_1,R}\|_{L^4(P_R(z_1))}\lesssim \theta^2 \|u-u_{z_1,R}\|_{L^4(P_R(z_1))} + \|u^2\|_{L^4(P_R(z_1))}$ and $\|\nabla u^1\|_{L^2(P_{\theta R}(z_1))}\lesssim \theta^2 \|\nabla u\|_{L^2(P_R(z_1))} + \|\nabla u^2\|_{L^2(P_R(z_1))}$. Thus we obtain the following estimates for local growth of $|u-u_{z_1,R}|$ and $|\nabla u|$.
    \begin{equation}\label{est: interior u}\begin{split}
        \int_{P_{\theta R}(z_1)} |u-u_{z_1,\theta R}|^4 & \lesssim ( \theta^8 + \|u\|^4_{L^4(P_R(z_1))} )\int_{P_R(z_1)} |u-u_{z_1,R}|^4 + \|\nabla d\|_{L^4(P_R)}^8 + \| \nabla P \|^4_{L^{\frac43}(P_R)}, \\
        \Big( \int_{P_{\theta R}(z_1)} |\nabla u|^2 \Big)^2 & \lesssim \theta^8 \Big( \int_{P_R(z_1)} |\nabla u|^2 \Big)^2 + \|u\|^4_{L^4(P_R)} \int_{P_R(z_1)} |u-u_{z_1,R}|^4 + \|\nabla d\|^8_{L^4(P_R)} + \|\nabla P\|^4_{L^{\frac43}(P_R)}.
    \end{split}\end{equation}

    \emph{Step 4:} We need to estimate $\|\nabla P\|_{L^{\frac43}}$. Again, we separate $P=P^1+P^2$ where $P^1: B_R(x_1)\to\RR$ satisfies Poisson's equation $\Delta P^1 = - \nabla\cdot(u\cdot \nabla u + \nabla\cdot (\nabla d\odot \nabla d-\frac12|\nabla d|^2\Id_2))$ with zero boundary condition $P^1=0$ on $\partial B_R(x_1)$, and harmonic function $P^2$ satisfies Laplace equation with Dirichlet boundary condition $P^2=P$ on $\partial B_R(x_1)$.  The Calderon-Zygmund theory, together with  interior  $W^{2,1}_2$-estimate   $\|\nabla^2 d\|_{L^2(P_{\frac34})}\lesssim \|u\|_{L^4(P_1)} + \|\nabla d\|_{L^4(P_1)}$ for $\eqref{LCF}_3$, gives 
    \[\|\nabla P^1\|_{L^{\frac43} (P_{\frac{R}2}(z_1))} \lesssim \|u\|_{L^4(P_R(z_1))} \|\nabla u\|_{L^2(P_R(z_1))} + (\|u\|_{L^4(P_R(z_1))} + \|\nabla d\|_{L^4(P_R(z_1))})\|\nabla d\|_{L^4(P_R(z_1))}. \]
    We apply Harnack inequality for $\nabla P^2$ and use the above estimate to get
    \begin{equation}\label{est: interior P} \int_{P_{\theta R}(z_1)} |\nabla P|^{\frac43} \lesssim \theta^2 \int_{P_{{R}}} |\nabla P|^{\frac43} + \|u\|_{L^4(P_R)}^{\frac43}  \|\nabla u\|^{\frac43}_{L^2(P_R)} + \Big(\|u\|_{L^4(P_R)}^{\frac43} + \|\nabla d\|^{\frac43}_{L^4(P_R)} \Big)\|\nabla d\|^{\frac43}_{L^4(P_R)}. \end{equation}

    \emph{Step 5:} Adding the above estimates \eqref{est: interior growth d},\eqref{est: interior growth dt}, \eqref{est: interior u}, and \eqref{est: interior P} for $\nabla d, d_t, \nabla u, \nabla P$ together, we have
    \[ A_{\theta R} \lesssim (\theta^6+\|u\|^4_{L^4(P_R)}) A_R + (\|u\|^4_{L^4(P_R)} + \|\nabla d\|^4_{L^4(P_R)})\|\nabla d\|^4_{L^4(P_R)} \lesssim \theta^{4+4\alpha_1} A_R +  R^{4+4\alpha_1}\]
    where $A_r=\|u-u_{z_r}\|^4_{L^4(P_r)} + \|\nabla u\|^4_{L^2(P_r)} + \|\nabla P\|^4_{\frac{4}{3}(P_R)}$, and we take $\alpha_1=\frac{1+\alpha}4,\epsilon_0\leq \theta^{\frac32}\leq R\leq \theta$. The characterization of Campanato spaces states that $u\in C^{\alpha_1}(P_{\frac12})$.
    \end{proof}

\begin{lemma}[Boundary smallness]\label{lem:boundary_smallness} For any $\alpha\in(0,1)$, there exists $\epsilon_0$ such that if for $z_0=(x_0,t_0)\in\partial\Omega_+\times\RR$ and $r>0, (u,d)\in W_2^{1,0}(P_{2r}^+(z_0)), P\in W^{1,\frac43}(P_{2r}^+(z_0))$ is a weak solution of \eqref{LCF}, with $(u,d)$ satisfies \eqref{PFB} on $\Gamma_{2r}^+(x_0)\times[t_0-r^2,t_0]$
and
\begin{equation}\label{smallness_boundary}
    \int_{P_r^+(z_0)} |u|^4 + |\nabla d|^4 \leq \epsilon_0^4,
\end{equation}
then $(u,d)\in C^\alpha(P_{\frac r2}^+(z_0))$. Moreover,
\begin{align}
    \label{boundary_holder_d} [d]_{C^\alpha(P_{\frac r2}^+ (z_0))} & \lesssim \|u\|_{L^4(P_r^+(z_0))} + \|\nabla d\|_{L^4(P_r^+(z_0))}, \\
    \label{boundary_holder_u} [u]_{C^\alpha(P_{\frac r2}^+(z_0))} & \lesssim \|u\|_{L^4(P_r^+(z_0))} + \|\nabla d\|_{L^4(P_r^+(z_0))} + \|\nabla P\|_{L^{\frac 43}(P_r^+(z_0))}.
\end{align}
\end{lemma}

\begin{proof}

     \emph{Step 1:} First we estimate boundary  growth for $\nabla d$. Consider $z_1=(x_1,t_1)\in \Gamma_{\frac12}\times[-\frac14,0]$ and $0<R\leq\frac14$. We separate $d=d^1+d^2$, where $d^1: P_R^+(z_1)\to\RR^3$ solves homogeneous heat equation $\partial_t d^1 -\Delta d^1 = 0  \text{ in } P_R^+(z_1)$ with mixed boundary condition $d^1 = d    \text{ on } S_R(x_1)\times[t_1-R^2,t_1], d^1 = d  \text{ on } B_R^+(x_1)\times\{t=t_1-R^2\},$ and $d^1 \text{ satisfies } \eqref{PFB}  \text{ on } \Gamma_R(x_1)\times[t_1-R^2,t_1]$. And $d^2=d-d^1$ solves inhomogeneous heat equation $   \partial_t d^2 -\Delta d^2 = -u\cdot\nabla d + |\nabla d|^2d  \text{ in } P_R^+(z_1) $ with mixed boundary condition $ d^2 = 0  \text{ on } B_R^+(x_1)\times\{t=t_1-R^2\},  d^2=0  \text{ on } S_R(x_1)\times[t_1-R^2,t_1]$, and $ d^2 \text{ satisfies } \eqref{PFB}  \text{ on } \Gamma_R(x_1)\times[t_1-R^2,t_1]$. Same as in Lemma \ref{lem:local_smallness}, we test equation  with  $\Delta d^2$, and in the end we obtain
     \[   \int_{P_R^+(z_1)} |\nabla d^2|^4 \lesssim \Big( \int_{P_R^+(z_1)} |u|^4 + |\nabla d|^4 \Big)\int_{P_R^+(z_1)} |\nabla d|^4. \]
 We apply Lemma \ref{heat: gradient} to components of $d^1$ and follow the same argument in Lemma \ref{lem:local_smallness} to obtain
    \begin{equation}\label{est: boundary d} \int_{P_{\theta R}^+(z_1)} |\nabla d|^4 \leq  C(\theta^4 + \epsilon_0^4) \int_{P_R^+(z_1)}|\nabla d|^4 \leq \theta_0^{4\alpha}\int_{P_R^+(z_1)}|\nabla d|^4 \end{equation}
    given that $\theta\leq \theta_0$ with $2C\theta_0^4\leq \theta_0^{4\alpha}$ and $\epsilon_0\leq\theta_0$. \\

    \emph{Step 2:} Then we estimate  boundary  growth for $|d_t|$. Because  components of $d_t$ satisfy either Neumann boundary condition or zero Dirichlet boundary condition, we test $\eqref{LCF}_3$ by $d_t\phi^2$ where the cut-off function $\phi \in C_0^\infty (B_R(x_1))$ is such that $0\leq\phi\leq 1$, $\phi\equiv1$ on $B_{\frac{r}2}(x_1)$, and $|\nabla \phi|\lesssim\frac1r$. The same computation as in Lemma \ref{lem:local_smallness} yields
    \begin{equation}\label{est: boundary dt} \int_{P_{\frac{r}2}^+(z_1) } |d_t|^2 \lesssim r^2 + \Big(\frac{r}R \Big)^{2\alpha} \|\nabla d\|_{L^4(P_R^+(z_1))}^2. \end{equation}
    Estimation \eqref{est: boundary d} and \eqref{est: boundary dt}, together with parabolic Morrey's lemma, imply that $d\in C^\alpha(P^+_{\frac12})$.

    \emph{Step 3:} Now we estimate the boundary local growth for $|\nabla u|$.  Let $H\subseteq\RR^2$ be a bent half-plane such that $P_R^+(z_1)\subseteq H$ and $\Gamma_R(z_1)\subseteq\partial H$. Let $u^1: H \times[-1,0]\to\RR^2$ solve  non-homogeneous, non-stationary Stokes equation $\partial_t u^1-\Delta u^1 + \nabla P^1 = -(u\cdot\nabla u + \nabla\cdot(\nabla d\odot \nabla d- \frac12|\nabla d|^2\Id_2))\chi_{P_R^+}$ with Navier perfect-slip boundary condition ($u^1$ satisfies $\eqref{PFB}_1$ on $\partial H$) and zero initial   condition. Then    $u^2=u-u^1: P_R^+(z_1)\to\RR^2$ solves homogeneous, non-stationary Stokes equation $\partial_t u^2-\Delta u^2  + \nabla P^2=0$ with partially  free boundary condition ($u^2$ satisfies $\eqref{PFB}_1$ on $\Gamma_R(z_1)$) and initial condition $u^2=u$ at $t=t_1-R^2$. 
    
    We apply parabolic Sobolev inequality Lemma \ref{Sobolev embedding} to $u^1$ and use $W^{2,1}_{\frac43}$-estimates for $u^1$ (cf. \cite[Theorem 2.3]{BKP13}), together with H\"older inequality to obtain
    \[ \|\nabla u^1\|_{L^2(P_R^+)} + \|u^1\|_{L^4(P_R^+)} + \|\nabla P^1\|_{L^{\frac43(P_R^+)}} \lesssim \|u\|_{L^4(P_R^+)}\|\nabla u\|_{L^2(P_R^+)} + \|\nabla^2d\|_{L^2(P_R^+)}\|\nabla d\|_{L^4(P_R^+)}. \]
We apply Lemma \ref{lem: W21pq} to $u^2$ and have, for any $q>4$,
    \begin{align*}
     &  R^{\frac{3}{2} -\frac{2}{q}}  \Big( \|u^2\|_{W^{2,1}_{q,\frac43} (P^+_{\frac{R}2}(z_1))} + \|\nabla P^2\|_{L^{\frac43(P_R^+(z_1))}} \Big)   \lesssim \|\nabla u^2\|_{L^2(P_R^+(z_1))} + \|\nabla P^2\|_{\frac43(P_R^+(z_1))} \\
     & \lesssim   \|\nabla u\|_{L^2(P_R^+(z_1))} + \|\nabla P\|_{L^{\frac43}(P_R^+(z_1))} + \|u\|_{L^4(P_R^+(z_1))} \|\nabla u\|_{L^2(P_R^+(z_1))} + \|\nabla^2 d\|_{L^2)P_R^+(z_1)} \|\nabla d\|_{L^4(P_R^+(z_1))}.
     \end{align*}
 Moreover, we use the Sobolev inequality and H\"older inequality to see that for any $\theta\in(0,\frac14)$,
     \begin{align*}
     & \|u^2-u^2_{z_1,\theta R}\|_{L^4(P^+_{\theta R}(z_1) ) } + \|\nabla P^2\|_{L^{\frac43}(P^+_{\theta R}(z_1))} \lesssim (\theta R)^{\frac32-\frac2q}\Big(\|u^2\|_{W^{2,1}_{q,\frac43} (P^+_{\frac{R}2}(z_1))} + \|\nabla P^2\|_{L^{\frac43(P_R^+(z_1))}}  \Big) \\
     & \lesssim \theta^{\frac32-\frac2q} \Big(  \|\nabla u\|_{L^2(P_R^+)} + \|\nabla P\|_{L^{\frac43}(P_R^+)} + \|u\|_{L^4(P_R^+)} \|\nabla u\|_{L^2(P_R^+)} + \|\nabla^2 d\|_{L^2(P_R^+)} \|\nabla d\|_{L^4(P_R^+)} \Big).
     \end{align*}

     To estimate $\|\nabla u^2\|_{L^2} $, let $\phi\in C_0^1(B_{\theta R}(x_1))$ be such that $0\leq \phi\leq 1, \phi\equiv 1$ on $B_{\frac{\theta R}{2}}(x_1)$, and $|\nabla \phi|\leq \frac2{\theta R}$. Multiplying the equation of $u^2$ by $(u^2-u^2_{z_1,\theta R})\phi^2$ and integrating over $B_{\theta R}^+(x_1)$, we obtain
     \[ \frac{d}{dt}\int_{B^+_{\theta R}(x_1)} |u^2- u^2_{z_1,\theta  R}|^2\phi^2 + \int_{B^+_{\theta R}(x_1)}|\nabla u^2|\phi^2 \lesssim \frac1{(\theta R)^{2}} \int_{B_{\theta R}^+ } |u^2-u^2_{z_1,\theta R}|^2 + \int_{B^+_{\theta R}}|\nabla P^2||u^2-u^2_{z_1,\theta R}|. \]
    Integrating over $[s_0,t_1]$, where $s_0\in[t_1-(\theta R)^2,t_1-\frac12(\theta R)^2]$ is such that 
    \[ \int_{B^+_{\theta R}(x_1)\times\{s_0\} } |u^2-u^2_{z_1,\theta R}|^2 \lesssim (\theta R)^{-2} \int_{P^+_{\theta R}(z_1)} |u^2-u^2_{z_1,\theta R}|^2, \]
    we obtain
    \begin{align*}
       & \int_{P^+_{\frac{\theta R}{2}}(z_1)} |\nabla u^2|^2 \lesssim (\theta R)^{-2} \int_{P_{\theta R}^+ (x_1)} |u^2-u^2_{z_1,\theta R}|^2 + \int_{P^+_{\theta R}(x_1)}|\nabla P^2||u^2-u^2_{z_1,\theta R}| \\
      & \lesssim  \| u^2 - u^2_{z_1,\theta R} \|^2_{L^4(P^+_{\theta R} (z_1))}  +  \| \nabla P^2 \|_{L^{\frac43} (P^+_{\theta R} (z_1))} \|u^2-u^2_{z_1,\theta R}\|_{L^4(P^+_{\theta R}(z_1))} \\
      & \lesssim  \| u^2 - u^2_{z_1,\theta R} \|^2_{L^4(P^+_{\theta R} (z_1))}  +  \| \nabla P^2 \|^2_{L^{\frac43} (P^+_{\theta R} (z_1))} .
    \end{align*}

    Collecting that above estimates, and we set $q=8$, we arrive
    \begin{align*}  
    & \|u-u_{z_1,\theta R}\|_{L^4(P^+_{\theta R}(z_1))} + \|\nabla u\|_{L^2(P^+_{\theta R}(z_1))} + \|\nabla P\|_{L^{\frac43}(P^+_{\theta R}(z_1) )} \\ 
      \lesssim & \Big( \theta^{\frac54} + \|u\|_{L^4(P_R^+(z_1))} \Big) \Big( \|u-u_{z_1,R}\|_{L^4(P_R^+(z_1))} + \|\nabla u\|_{L^2({P^+_R(z_1)})} + \|\nabla P\|_{L^{\frac43}(P^+_R(z_1))} \Big) \\ 
    &  +  \|\nabla^2 d\|_{L^2(P^+_{\frac{R}{2}}(z_1))}\|\nabla d\|_{L^4(P_R^+(z_1))} ,
    \end{align*}
and we have $W^{2,1}_2$-estimate for $d$ that
    \[ \|\nabla^2 d\|_{L^2(P^+_{\frac{R}2}(z_1))} \lesssim \|u\|_{L^4(P_R^+(z_1))} + \|\nabla d\|_{L^4(P_R^+(z_1))}. \]

    \emph{Step 4:} By choosing $\theta=\theta_0$ sufficiently small and $\epsilon_0\leq \theta_0$, we obtain for $ 0<r\leq\frac14 $,
    \[ \Theta^+(z_1,\theta R) \leq \theta \Theta^+(z_1,R) + C(\theta+R^\alpha)R^\alpha, \]
    where 
    \[ \Theta^+(z_1,r) := \|u-u_{z_1,r}\|_{L^4(P_r^+(z_1))} + \|\nabla u \|_{L^2(P_r^+(z_1))} + \|\nabla P\|_{L^{\frac43}(P_r^+(z_1))}.\]
We repeat the same argument as in Lemma \ref{lem:local_smallness} combined with the characterization of Campanato space and conclude that $u\in C^\alpha(P^+_{\frac12})$.
\end{proof}

 \subsection*{Proof of Theorem \ref{Thm A}}
\begin{proof}
    Assumption $u\in L^\infty([0,T], L^2(\Omega))\cap L^2([0,T], H^1(\Omega))$ and Ladyzhenskaya inequality give $u\in L^4(\Omega_T)$. Equation $\eqref{LCF}_3$ and the fact that $|d|=1$ give $d\Delta d + |\nabla d|^2=0$, which, together with assumption $\nabla d\in L^2([0,T], H^1(\Omega))$, gives $|\nabla d|\in L^4(\Omega_T)$ and hence $u\cdot\nabla u, \nabla\cdot(\nabla d\odot\nabla d-\frac12|\nabla d|^2\Id_2)\in L^{\frac43}(\Omega_T)$. It follows from absolute continuity of $\int |u|^4 + |\nabla d|^4$ that if $z_0=(x_0,t_0)\in\overline{\Omega}\times(0,T]$, there exists $r_0>0$ such that for $r<r_0$,
    \[ \int_{P_r(x_0)\cap \Omega_T} |u|^4 + |\nabla d|^4 \leq\epsilon_0^4, \]
    where $\epsilon_0$ is given in Lemma \ref{lem:local_smallness} and Lemma \ref{lem:boundary_smallness}. Hence, we deduce that $(u,d)\in C^\alpha(P_{\frac{r_0}2}(z_0)\cap\Omega\times(0,T])$. Consequently, we have $(u,d)\in C^\alpha(\overline{\Omega}\times(0,T])$.

    For the higher order regularity, we use Lemma \ref{lem: Riesz-Morrey}: for interior point $(x_0,t_0)\in\Omega\times(0,T)$, pick $r_0>0$ sufficiently small as above and also $P_{r_0}(x_0)\subseteq \Omega\times(0,T)$. Previous estimates give $\nabla d \in M^{2,2-2\alpha}(P_{\frac{r_0}{2}}(x))$ for any $\alpha\in(0,1)$. Consider $\tilde{d}=d\phi: P_{r_0}(x_0) \to \RR^3$ where the cutoff $\phi\in C_0^\infty(P_{r_0}(x_0)), 0\leq \phi\leq 1, \phi\equiv 1$ on $P_{\frac{r_0}2}(x_0)$, and $|\phi_t|, |\nabla \phi|, |\nabla^2\phi|\lesssim 1/r_0^2$. Then $\tilde{d}$ solves $\partial_t \tilde{d} - \Delta \tilde{d} = F$ where 
    \[ F=(\phi|\nabla d|^2d - d(\phi_t-\Delta\phi) - 2\nabla\phi\cdot\nabla d - u\cdot\nabla d \phi)\chi_{P_{r_0}(x_0)}\]

    Since $\nabla d \in L^4(\Omega_T)$, $u\in C^\alpha(\overline{\Omega}\times(0,T])$, and \eqref{est: interior growth d}, we can check that $F\in L^1(\RR^{n+1})$ and $F\in \widetilde{M}^{1,2-2\alpha}(\RR^{n+1})$. Because we have $\tilde{d}=G* F$ and thus $\nabla \tilde{d} = \nabla G* F$ where $G$ is the fundamental solution of the heat equation on $ \RR^n$, and $|\nabla G(x,t)|\lesssim \frac{1}{\delta((0,0), (x,t))^{1+n}}$, direct computation together with Lemma \ref{lem: Riesz-Morrey} yields
    \[ |\nabla \tilde{d}(z)|\lesssim \widetilde{I}_1(|F|)(z) \in L^{\frac{2-2\alpha}{1-2\alpha}, *}(\RR^{n+1}). \]
    Notice that $\frac{2-2\alpha}{1-2\alpha}\to+\infty$ as $\alpha\nearrow \frac12$. Thus we use interpolation of (weak) $L^p$ spaces to get  $\nabla \tilde{d}\in L^q(\RR^{n+1})$ for any $q>1$.
    Hence, $W^{2,1}_q$-estimate for the heat equation gives $\tilde{d}\in W^{2,1}_q(P_{r_0}(x_0))$ and thus $d\in W^{2,1}_q(P_{\frac{r_0}{2}}(x_0))$. Sobolev inequality then gives $\nabla   d\in C^\alpha(P_{\frac{r_0}2}(x_0))$ for any $\alpha\in(0,1)$.
     
    Consequently, we apply Schauder estimate to system $\eqref{LCF}_3$ to obtain $d\in C^{2,1}_\alpha(\Omega\times(0,T))$. This gives the H\"older regularity of external forces in $\eqref{LCF}_1$ and thus the standard $C^{2,1}_\alpha$-regularity theory implies that $u\in C^{2,1}_\alpha(\Omega\times(0,T))$. Starting from $C^{2,1}_\alpha$-regularity for pair $(u,d)$, we use the standard boot-strap argument to conclude that $(u,d)\in C^\infty(\Omega\times(0,T))$. Boundary $C^{2,1}_\alpha$-regularity for $(u,d)$ can be obtained in a similar way, by taking instead $\tilde{d}=(d-d')\phi$ where $d':P_r(x_0)^+\to\RR^3$ solves heat equation and  $d'=d$ on $ \Gamma_r(x_0)\times(t_0-t^2,t_0)\cup S_r(x_0)\times(t_0-t^2,t_0) $.
\end{proof}


\bigskip


\section{Existence of short time smooth solutions}

\medskip

 We would like to show that short time smooth solutions to \eqref{LCF}-\eqref{IC} exists for smooth initial and boundary data. Later in the proof of Theorem \ref{Thm B} we shall use approximation of smooth data and then show the existence. The argument to show the short-time existence in this section  also applies to the case of three dimensions.  

\begin{theorem}\label{short_time_existence}
For any $\alpha>0$, if $u_0\in C^{2,\alpha}(\overline{\Omega},\RR^2)\cap\Hbf$ and $d_0\in C^{2,\alpha}(\overline{\Omega},\SS^2)\cap\mathbf{J}$, then there exists $T>0$ depending on $\|u_0\|_{C^{2,\alpha}(\Omega)},\|d_0\|_{C^{2,\alpha}(\Omega)}$ such that there is a unique smooth solution $(u,d)\in C^{2,1}_\alpha(\overline{\Omega}\times[0,T),\RR^2\times\SS^2 )$ to the initial-boundary value problem \eqref{LCF}-\eqref{IC}.
    
\end{theorem}

\begin{proof}
   We will use the contraction mapping principle in the similar spirit as \cite{LLW10}. For $T>0$ and $K>0$ to be chosen later, denote 
    \[ X=\{ (v,f)\in C^{2,1}_\alpha (\overline{\Omega}_T, \RR^2\times\RR^3)| \nabla\cdot v=0, (v,f)|_{t=0}=(u_0,d_0), \|v-u_0\|_{C^{2,1}_\alpha(\Omega_T)} + \|f-d_0\|_{C^{2,1}_\alpha(\Omega_T)} \leq K \}. \]
    Equip $X$ with the norm
    \[ \|(v,f)\|_X:= \|v\|_{C^{2,1}_\alpha (\Omega_T)} + \|f\|_{C^{2,1}_\alpha (\Omega_T)}.  \]
    Then we can show that $(X,\|\cdot\|_X)$ is a Banach space. For any $(v,f)\in X$, let $(u,d)$ be the unique solution to the following mixed system (inhomogeneous Stokes system with initial condition and Navier-slip boundary condition; inhomogeneous heat equation with initial condition and Neumann boundary condition \ref{PFB}):
    \begin{equation}\label{Stokes system L}\begin{split}
        \partial_t u -\Delta u + \nabla P & = -v\cdot\nabla v - \nabla \cdot (\nabla d\odot \nabla d-\frac12|\nabla d|^2\Id_2)   \ \text{ in } \Omega_T 
        \\
        \nabla \cdot u &=0 \ \text{ in } \Omega_T \\
        \partial_t d -\Delta d & = |\nabla f|^2f -v\cdot \nabla f  \ \text{ in } \Omega_T \\
        (u,d) & = (u_0,d_0)   \ \text{ on } \Omega\times\{t=0\} \\
        (u,d) & \text{ satisfies } \ref{PFB}  \ \text{ on } \partial\Omega\times(0,T).
    \end{split}\end{equation}
    
    We define the operator
   $  L : X\to C^{2,1}_\alpha(\overline{\Omega}_T,\RR^2\times\RR^3) $ to be $L(v,f)=(u,d)$. In the following two lemmas, we are going to show that for $T>0$ sufficiently small and $K>0$ sufficiently large, $L:X\to X$ and $L$ is a contraction map. Then there exists a unique solution $(u,d)\in C^{2,1}_\alpha(\overline{\Omega}\times[0,T),\RR^2\times\RR^3 )$ to \eqref{LCF}-\eqref{assumption1}. It remains to show that $|d|=1$. In fact, this can be done by using the maximum principle to the equation for $|d|^2$.
\end{proof}

\begin{lemma} There exist $T>0$ and $K>0$ such that $L:X\to X$.
\end{lemma}
\begin{proof} For any $(v,f)\in X$. Let $(u,d)=L(v,f)$ be the unique solution to system \eqref{Stokes system L}. We use  $C_0>0$ to denote constants depending only on $\|u_0\|_{C^{2,\alpha}}$ and $\|d_0\|_{C^{2,\alpha}}$.

\emph{Step 1:} Estimation of $\|d-d_0\|_{C^{2,1}_\alpha(\Omega_T)}$. 
Assume $K\geq C_0$. By the Schauder theory of parabolic systems, we have
\begin{equation}\label{est: Schauder d} \|d-d_0\|_{C^{2,1}_\alpha(\Omega_T)} \lesssim \|v\cdot\nabla f\|_{C^\alpha(\Omega_T)} + \||\nabla f|^2f\|_{C^{\alpha}(\Omega_T)} .\end{equation}
    And we can estimate the first term in \eqref{est: Schauder d} as
    \begin{align*}
       &  \|v\cdot\nabla f\|_{C^\alpha(\Omega_T)}   \leq \|v\cdot\nabla f - u_0\cdot\nabla d_0\|_{C^\alpha(\Omega_T)} + \|u_0\cdot\nabla d_0\|_{C^\alpha(\Omega)} \\
        & \leq \|(v-u_0)\cdot\nabla f\|_{C^\alpha(\Omega_T)} + \|u_0\cdot\nabla (f-d_0)\|_{C^\alpha(\Omega_T)} + C_0 \\
        & \leq 2K(\|v-u_0\|_{C^0(\Omega_T)} + \|v-u_0\|_{C^\alpha(\Omega_T)} ) + C_0(1+ \|\nabla (f-d_0)\|_{C^0(\Omega_T)} + \|\nabla (f-d_0)\|_{C^\alpha(\Omega_T)}).
    \end{align*}
Since $v-u_0=f-d_0=0$ at $t=0$, it is easy to see
    \[ \|v-u_0\|_{C^0(\Omega_T)} \leq KT, \  \|\nabla(f-d_0)\|_{C^0(\Omega_T)} \leq KT . \]
Meanwhile, we use interpolation inequality (see \cite[Lemma 6.32]{GT01}) to control terms with $C^{\alpha}$-norm:
    \begin{align*}
        \|v-u_0\|_{C^\alpha(\Omega_T)} \lesssim \frac1\delta \|v-u_0\|_{C^0(\Omega_T)} + \delta\|v-u_0\|_{C^{2,1}_{\alpha}(\Omega_T)} \lesssim (\delta + \frac1\delta)K, \\
        \|\nabla(d-f)\|_{C^\alpha(\Omega_T)} \lesssim\frac1\delta \|\nabla(d-f)\|_{C^0(\Omega_T)}  + \delta\|d-f\|_{C^{2,1}_\alpha(\Omega_T)} \lesssim (\delta+\frac1\delta)K.
    \end{align*}
    Collecting the above estimates together, we obtain
    \[ \|v\cdot\nabla f\|_{C^\alpha(\Omega_T)}  (C_0K + CK^2)(T+\delta+\frac{T}\delta) + C_0\leq \frac{\sqrt{K}}{4}\]
    with $K=16 C_0^2, \delta \lesssim ((C_0+ C^2K)\sqrt{K})^{-1}$, and $T=\delta^2$.

    Then we are going to estimate the second term in \eqref{est: Schauder d}. Notice that 
    \begin{align*}
        & \||\nabla f|^2f\|_{C^\alpha(\Omega_T)} \leq \||\nabla f|^2f - |\nabla d_0|^2 d_0\|_{C^\alpha(\Omega_T)} + \||\nabla d_0|^2d_0\|_{C^\alpha(\Omega_T)} \\
          \leq & \||\nabla f|^2 (f-d_0)\|_{C^\alpha(\Omega_T)} + \|d_0(|\nabla f|^2-|\nabla d_0|^2)\|_{C^\alpha(\Omega_T)} + C_0 := I_1 + I_2 + C_0,
    \end{align*}
    where $I_1$ can be estimated by
    \begin{align*}
       & I_1 \leq \|f-d_0\|_{C^\alpha(\Omega_T)}\|\nabla f\|^2_{C^0(\Omega_T)} + \|f-d_0\|_{C^0(\Omega_T)} \|\nabla f\|^2_{C^\alpha(\Omega_T)} \\
       \lesssim & K^2(\|f-d_0\|_{C^0(\Omega_T)} + \|f-d_0\|_{C^\alpha(\Omega_T)}) \\
       \lesssim & K^2\Big((1+\frac1\delta)\|f-d_0\|_{C^0(\Omega_T)} + \delta\|f-d_0\|_{C^{2,1}_\alpha(\Omega_T)} \Big) \lesssim K^3(\frac{T}\delta + \delta),
    \end{align*}
    and $I_2$ can be estimated by
    \begin{align*}
        & I_2 \leq \||\nabla f|^2- |\nabla d_0|^2\|_{C^\alpha(\Omega_T)} + \||\nabla f|^2 - |\nabla d_0|^2\|_{C^0(\Omega_T)}\|d_0\|_{C^\alpha(\Omega_T)} \\
        \leq & \| (|\nabla f| + |\nabla d_0|)|\nabla (f-d_0)| \|_{C^\alpha(\Omega_T) } + C_0 \|(|\nabla f|+ |\nabla d_0|)|\nabla(f-d_0)| \|_{C^0(\Omega_T)} \\
        \lesssim & (1+C_0) K (\|\nabla(f-d_0)\|_{C^0(\Omega_T)} + \|\nabla (f-d_0)\|_{C^\alpha(\Omega_T)} ) \lesssim (1+C_0)K^2(T+\delta+\frac{T}{\delta}).
    \end{align*}
Therefore, we collect the above estimates together and conclude
    \[ \||\nabla f|^2f\|_{C^\alpha(\Omega_T)} \lesssim K^3(\frac{T}\delta + \delta) + (1+C_0)K^2(T+\delta+\frac{T}{\delta})\]
    with $K=16C_0^2, \delta\lesssim((1+C_0)K^{\frac52})^{-1}$, and $T=\delta^2$.

    As a consequence, we can simplify \eqref{est: Schauder d} to be 
    \begin{equation}\label{est: Schauder d sim}
        \|d-d_0\|_{C^{2,1}_\alpha(\Omega_T)} \leq \frac{\sqrt{K}}{2}
    \end{equation}

    \emph{Step 2:} We estimate $\|u-u_0\|_{C^{2,1}_\alpha(\Omega_T)}$ is a similar way. By the Schauder theory for non-homogeneous, non-stationary Stokes equations, we have
    \begin{equation}\label{est: Schauder u}
        \|u-u_0\|_{C^{2,1}_\alpha(\Omega_T)} \lesssim \|v\cdot\nabla v\|_{C^\alpha(\Omega_T)} + \|\nabla\cdot(\nabla d\odot\nabla d - \frac12|\nabla d|^2\Id_2)\|_{C^\alpha(\Omega_T)}
    \end{equation}
and estimate the first term in \eqref{est: Schauder u} as
    \begin{align*}
        & \|v\cdot \nabla v\|_{C^{\alpha(\Omega_T)}} \leq \|(v-u_0)\cdot\nabla v\|_{C^\alpha(\Omega_T)}  + \|u_0\cdot\nabla (v_0-u_0)\|_{C^\alpha(\Omega_T)} + \|u_0\cdot\nabla u_0\|_{C^\alpha(\Omega_T)}  \\
        \leq & K( \|v-u_0\|_{C^0(\Omega_T)} + \|v-u_0\|_{C^\alpha(\Omega_T)} ) + C_0(\|\nabla(v-u_0)\|_{C^0(\Omega_T)} + \|\nabla(v-u_0)\|_{C^\alpha(\Omega_T)} + C_0 ) \\
        \lesssim & (C_0K + K^2) (T+ \delta + \frac{T}{\delta}) + C_0 \leq \frac{K}{4}
    \end{align*}
    with $K=8C_0, \delta\lesssim ((1+C_0)K)^{-1}$, and $T=\delta^2$.

For the second term in \eqref{est: Schauder u}, we have
    \begin{align*}
        & \|\nabla\cdot(\nabla d\odot\nabla d - \frac12|\nabla d|^2\Id_2)\|_{C^\alpha(\Omega_T)} \leq \||\nabla^2(d-d_0)| |\nabla d|\|_{C^\alpha(\Omega_T)} + \||\nabla^2d_0||\nabla(d-d_0)|\|_{C^\alpha(\Omega_T)} \\
        & + \||\nabla^2d_0||\nabla d_0|\|_{C^\alpha(\Omega_T)}   \leq C_0 + \|d-d_0\|_{C^{2,1}_\alpha(\Omega_T)} \|d\|_{C^{2,1}_\alpha(\Omega_T)} + C_0\|d-d_0\|_{C^{2,1}_\alpha(\Omega_T)} \\
        & \leq  C_0 + \frac{\sqrt K}{2} (C_0 + \frac{\sqrt{K}}{2} ) + C_0\sqrt{K} \leq \frac{K}{2}.
    \end{align*}

    Combining the above estimates we have $\|u-u_0\|_{C^{2,1}_\alpha (\Omega_T)}\leq \frac34 K$, which together with \eqref{est: Schauder u}, yields
    \[ \|u-u_0\|_{C^{2,1}_\alpha(\Omega_T) } + \|d-d_0\|_{C^{2,1}_\alpha(\Omega_T)} \leq K. \]
    Consequently, $L:X\to X$.
\end{proof}

\begin{lemma} There exist sufficiently large $K>0$ and sufficiently small $T>0$ such that $L: X\to X$ is a contraction map.
\end{lemma}
\begin{proof}
    For any $(v_i,f_i)\in X, i=1,2$, let $(u_i,d_i)= L(v_i,f_i)\in X$. Denote $u=u_1-u_2, d=d_1-d_2, P=P_1-P_2, v=v_1-v_2, f=f_1-f_2$. Then $(u,d)$ solves the following mixed system:
    \begin{equation}\label{mixed system dif}\begin{split}
        \partial_t u - \Delta u + \nabla P & = G \ \ \text{ in } \Omega_T \\
        \nabla \cdot u & = 0\ \ \text{ in } \Omega_T \\
        \partial_t d -\Delta d & = H \ \ \text{ in } \Omega_T \\
        (u,d)&=(0,0)  \ \ \text{ on } \Omega\times\{t=0\} \\
        (u,d) & \text{ satisfies } \eqref{PFB} \ \ \text{ on } \partial\Omega\times(0,T),
    \end{split}\end{equation}
 where 
 \[ G= -(v\cdot \nabla v_1 + v_2\cdot\nabla v) - \nabla\cdot(\nabla d\odot\nabla d_1 + \nabla d_2\odot\nabla d),\]
 and
 \[ H = |\nabla f_1|^2 f + \nabla(f_1+f_2)\cdot\nabla f f_2 - (v\cdot\nabla f_1 - v_2\cdot\nabla f).\]
 By the previous lemma, we know that for $i=1,2$,
 \[ \|u_i-u_0\|_{C_\alpha^{2,1}(\Omega_T)} + \|d_i-d_0\|_{C^{2,1}_\alpha(\Omega_T)} \leq K. \]

 We apply the Schauder theory of parabolic systems for $\eqref{mixed system dif}_3$, and we have
 \begin{equation}\label{est: Schauder dif d}\begin{split}
     & \|d\|_{C^{2,1}_\alpha(\Omega_T)} \lesssim \|H\|_{C^\alpha(\Omega_T)} \lesssim \||v\cdot\nabla f_1| + |v_2\cdot\nabla f| +|\nabla f_1|^2|f| + |f_2|(|\nabla f_1|+ |\nabla f_2|)|\nabla f| \|_{C^\alpha(\Omega_T)} \\
     & \lesssim  K^2(\|v\|_{C^\alpha(\Omega_T)} + \|f\|_{C^\alpha(\Omega_T)} + \|\nabla f\|_{C^\alpha(\Omega_T)}  ) \\
     & \lesssim   K^2\Big(\delta(\|v\|_{C^{2,1}_\alpha(\Omega_T)}  + \|f\|_{C^{2,1}_\alpha(\Omega_T)}) + \frac1\delta(\|v\|_{C^\alpha(\Omega_T)} + \|f\|_{C^\alpha(\Omega_T)} )  \Big) \\
    & \lesssim   K^2(\delta + \frac{T}{\delta})(\|v\|_{C^{2,1}_\alpha (\Omega_T)} + \|f\|_{C^{2,1}_\alpha (\Omega_T)} ),
 \end{split}\end{equation}
 where we have used $\|v\|_{C^\alpha(\Omega_T)} + \|f\|_{C^\alpha(\Omega_T)} \lesssim (\|v\|_{C^{2,1}_\alpha(\Omega_T)} + \|f\|_{C^{2,1}_\alpha (\Omega_T)} )T$ which originates from the fact that $v=d=0$ at $t=0$.

 We apply the Schauder theory of non-homogeneous, non-stationary Stokes system for $\eqref{mixed system dif}_{1}$-$\eqref{mixed system dif}_2$, and we have
 \begin{equation}\label{est: Schauder dif u}\begin{split}
     & \|u\|_{C^{2,1}_\alpha(\Omega_T)} \lesssim \|G\|_{C^\alpha(\Omega_T)} \lesssim \| |v||\nabla v_1| + |v_2||\nabla v| + |\nabla^2d||\nabla d_1| + |\nabla^2d_2||\nabla d|\|_{C^\alpha(\Omega_T)} \\
     & \lesssim K\|d\|_{C^{2,1}_\alpha(\Omega_T)} + K(\|v\|_{C^\alpha(\Omega_T)} + \|\nabla v\|_{C^\alpha(\Omega_T)} ) \\
     & \lesssim K^3(\delta + \frac{T}{\delta})(\|v\|_{C^{2,1}_\alpha(\Omega_T)} + \|f\|_{C^{2,1}_\alpha(\Omega_T)} ) + K^2(\delta + \frac{T}{\delta})\|v\|_{C^{2,1}_\alpha(\Omega_T)} \\
     & \lesssim K^3(\delta + \frac{T}{\delta})(\|v\|_{C^{2,1}_\alpha(\Omega_T)} + \|f\|_{C^{2,1}_\alpha(\Omega_T)}).
 \end{split}\end{equation}

 It follows from \eqref{est: Schauder dif d} and \eqref{est: Schauder dif u} that 
 \[ \|L(v_1,f_1) - L (v_2,f_2)\|_{X} \lesssim K^3(\delta + \frac{T}{\delta})(\|v\|_{C^{2,1}_\alpha(\Omega_T)} + \|f\|_{C^{2,1}_\alpha(\Omega_T)}) \frac12\|(v_1,f_1)-(v_2,f_2)\|_X\]
 provided that $\delta$ and $T$ are sufficiently small. Then $L:X\to X$ is a contraction map as desired.
\end{proof}


\bigskip

\section{Energy estimation}

 \begin{lemma}[Global Energy]\label{lem: global energy}
 
 \medskip
 
For $0<t<+\infty$, suppose $u\in L^{2,\infty}(\Omega\times[0,T])\cap W_2^{1,0}(\Omega_T), d\in L^\infty([0,T],H^1(\Omega))\cap L^2([0,T], H^2(\Omega))$, and $\nabla P\in L^{\frac 43}(\Omega_T)$ is a weak solution to \eqref{LCF}-\eqref{assumption1}. Then, for any $0<t\leq T$, we have
\begin{equation}\label{global_energy}  \int_\Omega (|u|^2 + |\nabla d|^2)(\cdot,t) +  \int_{\Omega_t}(4| \DD u|^2 + 2|\Delta d+|\nabla d|^2d|^2) = \int_{\Omega} (|u_0|^2 + |\nabla d_0|^2). \end{equation}
\end{lemma}

\begin{proof}
     Partially free boundary condition \eqref{PFB} permits
     \begin{align*}
     &\int_\Omega(u\cdot\nabla u)\cdot u=0, \quad \int_\Omega \nabla P\cdot u = 0, \quad \int_{\Omega} d_t\cdot\Delta d=-\frac{d}{dt}\int_\Omega\frac12|\nabla d|^2\\
     &\int_{\partial\Omega} (\n\cdot\nabla)d\cdot (u\cdot\nabla)d =0, \quad \int_\Omega (u\cdot\nabla)(|\nabla d|^2)=0.
     \end{align*}
Assumption $u\in L^\infty([0,T], L^2(\Omega))\cap L^2([0,T], H^1(\Omega))$ and Ladyzhenskaya inequality yield $u\in L^4(\Omega_T)$. Equation $\eqref{LCF}_3$ and the fact that $|d|=1$ gives $d\Delta d + |\nabla d|^2=0$, which, together with assumption $\nabla\in L^2([0,T], H^1(\Omega))$, gives $|\nabla d|\in L^4(\Omega_T)$. Now we  test $\eqref{LCF}_1$ with $u$ and get
     \[ \frac{1}{2}\frac{d}{dt} \int_\Omega |u|^2 + 2\int_\Omega |\DD u|^2 = \int_{\Omega}  \nabla d\odot \nabla d  :\nabla u. \]
     where the operator $:$ stands for the inner product of two matrices.

     Testing $\eqref{LCF}_3$ with $\Delta d+|\nabla d|^2d$, we obtain
     \[ \int_\Omega (d_t + u\cdot\nabla d )\cdot\Delta d = \int_\Omega |\Delta d + |\nabla d|^2 d|^2. \]
     Further, we can compute that
     \[ \int_\Omega \Delta d\cdot (u\cdot\nabla d) = -\int_\Omega \nabla d: \nabla((u\cdot\nabla )d) = -\int_\Omega \nabla d\odot \nabla d : \nabla u-(u\cdot\nabla)|\nabla d|^2 = -\int_\Omega \nabla d\odot\nabla d : \nabla u.\]
     Adding the above equations together implies the global energy equality \eqref{global_energy}.
\end{proof}

\begin{lemma}[Interior and boundary energy inequalities]\label{lem: local energy}
For $0<t<+\infty$, suppose $u\in L^{2,\infty}(\Omega\times[0,T])\cap W_2^{1,0}(\Omega_T), d\in L^\infty([0,T],H^1(\Omega))\cap L^2([0,T], H^2(\Omega))$, and $\nabla P\in L^{\frac 43}(\Omega_T)$ is a weak solution to \eqref{LCF}-\eqref{assumption1}. Then, for any non-negative $\phi\in C^\infty(\Omega)$ and $0<s<t\leq T$, it holds that
\begin{equation}\label{local_energy}\begin{split}  &\int_\Omega \phi(|u|^2+|\nabla d|^2)(t) +  \int_s^t \int_\Omega \phi(4|\DD u|^2 + 2|\Delta d+|\nabla d|^2d|^2) \leq  \\
&\int_\Omega \phi(|u|^2 + |\nabla d|^2)(s) + C\int_s^t \int_\Omega |\nabla \phi|(|u|^3 + |P-P_\Omega||u| + |\DD u||u| + |\nabla d|^2|u| + |d_t||\nabla d|) \end{split}\end{equation}
where $P_\Omega$ is the average of $P$ over $\Omega$.    
\end{lemma}

\begin{proof} Testing $\eqref{LCF}_1$ by $u\phi$, we have
\[ \frac12\frac{d}{dt}\int_\Omega |u|^2\phi +  \int_\Omega  2 | \DD u|^2\phi =  \int_\Omega -u\cdot\nabla u \cdot u\phi -2 \DD u \cdot u\cdot\nabla\phi + (P-P_\Omega)\cdot u\cdot \nabla\phi + (\nabla d\odot \nabla d-\frac12|\nabla d|^2\Id_2):\nabla (u\phi), \]
and notice that
\begin{align*}
     \int_\Omega -u\cdot \nabla u\cdot u\phi &= \frac12\int_\Omega |u|^2 u\cdot\nabla \phi \\
    \int_\Omega (\nabla d\odot \nabla d - \frac12 |\nabla d|^2 \Id_2) : \nabla (u \phi) &= \int_\Omega (\nabla d\odot \nabla d):\nabla u \phi + \int_\Omega u\cdot\nabla d\cdot \nabla d\nabla \phi -\frac12 |\nabla d|^2 u\cdot\nabla \phi.
\end{align*}
Hence, $\eqref{LCF}_1$ gives energy estimates for $u$:
\begin{equation}\label{eng: est u} \frac12\frac{d}{dt}\int_\Omega |u|^2\phi +  \int_\Omega  2 | \DD u|^2\phi \leq \int_\Omega   (\nabla d\odot \nabla d):\nabla u \phi + (\frac12|u|^3 + |u||2\DD u| + |P-P_\Omega||u|+\frac32|\nabla d|^2|u|)|\nabla\phi|. \end{equation}

We test $\eqref{LCF}_3$ with $(\Delta d + |\nabla d|^2d)\phi$ to get
\[ \int_\Omega (d_t+u\cdot\nabla d)\cdot\Delta d \phi = \int_\Omega |\Delta d + |\nabla d|^2d|^2\phi \]
and notice that
\begin{align*}
    \int_\Omega d_t\cdot\Delta d\phi &= -\frac12\frac{d}{dt}\int_\Omega |\nabla d|^2\phi  -\int_\Omega d_t\cdot\nabla d\cdot\nabla \phi \\
    \int_\Omega u\cdot\nabla d \cdot\Delta d\phi & = -\int_\Omega \frac12 u^i\partial_i(|\nabla d|^2)\phi + \nabla d\odot\nabla d: \nabla u\phi  + (u\cdot\nabla)d(\nabla\phi\cdot\nabla d) \\
    & =\int_\Omega \frac12 |\nabla d|^2 u\cdot\nabla\phi - \nabla d\odot\nabla d:\nabla u\phi -(u\cdot\nabla d)(\nabla \phi\cdot\nabla d).
\end{align*}
Therefore, $\eqref{LCF}_3$ gives energy estimates for $\nabla d$:
\begin{equation}\label{eng: est d} \frac12\frac{d}{dt}\int_\Omega |\nabla d|^2\phi  + \int_\Omega |\Delta d + |\nabla d|^2d|^2\phi =\int_\Omega - d_t\cdot\nabla d\cdot\nabla\phi + \frac12|\nabla d|^2u\cdot\nabla\phi - \nabla d\odot\nabla d: \nabla u\phi - (u\cdot\nabla d)(\nabla\phi\cdot\nabla d).\end{equation}

Adding energy estimates \eqref{eng: est u} and \eqref{eng: est d} yields \eqref{local_energy}.
\end{proof}
 
\begin{remark}
    By virtue of Lemma \ref{lem: Korn} for non-axisymmetric domain, $\|\nabla u\|_{L^2(\Omega)}$ and $\|\DD u\|_{L^2(\Omega)}$ are equivalent norms, thus Lemma \ref{lem: global energy} and Lemma \ref{lem: local energy} can have corresponding inequalities with $|\DD u|^2$ replaced by $|\nabla u|^2$. We shall  use this variant in the proof of Theorem \ref{Thm B}, because we need to use the convergence of $L^2$-norm of  $\nabla u$ to zero in some domain, in order to conclude that $u\to 0$ in $H^1$. The original version of energy estimates in terms of $\DD u$ is not sufficient.
\end{remark}


\bigskip

\section{Global weak solution and proof of Theorem \ref{Thm B}}

\medskip
 
We now derive the life span estimate for smooth solutions in term of Sobolev space norms of initial data.
\begin{lemma}\label{lem: life span} Let $\epsilon_0>0$ be given in Lemma \ref{lem:local_smallness} and \ref{lem:boundary_smallness}.  There exist $0<\epsilon_1\ll \epsilon_0$, $0<R_0 \leq 1 $, and $\theta_0=\theta_0(\epsilon_1,E_0)\in(0,\frac14)$ with $E_0=\int_\Omega |u_0|^2+ |\nabla d_0|^2$, such that if $u_0\in C^{2,\beta}(\overline{\Omega},\RR^2)\cap\mathbf{H}$ and $d_0\in C^{2,\beta}(\overline{\Omega}, \SS^2)\cap\mathbf{J}$ satisfy
\begin{equation}
    \label{smallness_global} \sup_{x\in\overline{\Omega}}\int_{\Omega\cap B_{2R_0}(x)} |u_0|^2 + |\nabla d_0|^2 \leq \epsilon_1^2,
\end{equation}
then there exist $T_0\geq \theta_0R_0^2$ and a unique solution $(u,d)\in C^\infty(\Omega\times(0,T_0),\RR^2\times\SS^2)\cap C_\beta^{2,1}(\overline{\Omega} \times [0,T_0), \RR^2\times\SS^2)$ to \eqref{LCF}-\eqref{assumption1} satisfying
\begin{equation}
    \label{smallness_GSLT} \sup_{(x,t)\in \overline{\Omega}_{T_0}} \int_{\Omega\cap B_{R_0}(x)} (|u|^2+|\nabla d|^2)(\cdot,t)\leq 2\epsilon_1^2.
\end{equation}
\end{lemma}
\begin{proof}
    Theorem \ref{short_time_existence} states the existence of $T_0>0$ such that there exists a unique smooth solution $(u,d)\in C^\infty(\Omega\times(0,T_0),\RR^2\times\SS^2)\cap C^{2,1}_\alpha(\overline{\Omega}\times[0,T_0),\RR^2\times\SS^2)$ to \eqref{LCF}-\eqref{PFB}. Let $0<t_0\leq T_0$ be the maximal times such that 
    \begin{equation}\label{eqn: maximal time}
    \sup_{0\leq t\leq t_0}\sup_{x\in\overline{\Omega}} \int_{\Omega\cap B_{R_0}(x)} (|u|^2+|\nabla d|^2)(\cdot,t) \leq 2\epsilon_1^2.
    \end{equation}
    Since $t_0$ is defined to be  the maximal time, we have
    \[ \sup_{x\in\overline{\Omega}} \int_{\Omega\cap B_{R_0}(x)} (|u|^2+|\nabla d|^2)(\cdot,t_0) =2\epsilon_1^2 .\]
    Now we estimate the lower bound of $t_0$ as follows. Assume $t_0\leq R_0^2\leq 1$ (otherwise we have finished the proof). Set
    \begin{align*}
        E(t) & = \int_\Omega (|u|^2 + |\nabla d|^2)(\cdot,t) \\
        E_0  & = \int_\Omega (|u_0|^2 + |\nabla d_0|^2).
    \end{align*}
    Observe the energy inequality in Lemma \ref{lem: global energy}: for  $0<t\leq t_0$,
    \[ E(t) + \int_{\Omega_t} (|\DD u|^2 + |\Delta d + |\nabla d|^2d|^2) \leq E_0. \]
     Lemma \ref{lem: Korn} allows us to deduce that
    \[ \int_{\Omega_{t}} |\nabla u|^2 \lesssim \int_{\Omega_t} |u|^2 +  \int_{\Omega_t}  |\DD u|^2  \leq  (t_0+1) E_0\leq 2E_0 .\]

    Also, we use the Ladyzhenskaya inequality in Lemma \ref{Ladyzhenskaya}, and it follows that
    \[ \int_{\Omega_t} |\nabla d|^4 \lesssim \Ee^2_{R_0}(t)\Big(\int_{\Omega_t} |\Delta d|^2  + \frac{t E_0}{R_0^2} \Big), \]
    where 
    \begin{align*} 
    \Ee^2_{R_0}(t) & = \sup_{(x,s)\in\Omega_t} \int_{\Omega \cap B_{R_0}(x)} |\nabla d|^2(\cdot,s) \\
       \Ee^1_{R_0}(t) & = \sup_{(x,s)\in\Omega_t} \int_{\Omega \cap B_{R_0}(x)} |u|^2(\cdot,s) \\
    \Ee_{R_0}(t) & = \Ee_{R_0}^1(t) + \Ee_{R_0}^2(t)   .
    \end{align*}
    By  \eqref{eqn: maximal time}, we have $\Ee_{R_0}(t)\leq 4\epsilon_1^2$ for all $0\leq t\leq t_0$.
    As a consequence, 
    \[ \int_{\Omega_{t_0}} |\nabla d|^4 \lesssim \epsilon_1^2\Big( \int_{\Omega_{t_0}} |\Delta d|^2 + \frac{t_0 E_0}{R_0^2} \Big). \]
 The above inequality, together with the energy inequality, and the fact that $|\Delta d|^2\leq 2(|\Delta d + |\nabla d|^2d|+ |\nabla d|^4) $, imply
    \[ \int_{\Omega_{t_0}} |\Delta d|^2 \leq  E_0 + C_0\epsilon_1^2\Big( \int_{\Omega_{t_0}} |\Delta d|^2 + \frac{t_0 E_0}{R_0^2} \Big) .\]
Therefore by taking $0<\epsilon_1^2\leq \min(1,\frac{1}{2C_0})$, we have
\[
    \int_{\Omega_{t_0}} |\Delta d|^2 \leq C_0 E_0,
\]
and hence
\begin{equation}\label{est: L4-gradd} \int_{\Omega_{t_0}} |\nabla d|^4 \leq C_0\epsilon_1^2 E_0 .\end{equation}

Similarly, we apply Lemma \ref{Ladyzhenskaya} to $\int_{\Omega_{t_0}} |u|^4$ and get
\begin{equation}\label{est: L4-u}\begin{split}
    \int_{\Omega_{t_0}}|u|^4 &  \lesssim \Ee_{R_0}^1(t_0) \Big( \int_{\Omega_{t_0}} |\nabla u|^2 + \frac{1}{R_0^2}\int_{\Omega_{t_0}} |u|^2 \Big)  
    \lesssim \Ee_{R_0}^1(t_0) \Big( \int_{\Omega_{t_0}} |\nabla u|^2 + \frac{t_0E_0}{R_0^2} \Big) \\
    & \lesssim \epsilon_1^2(E_0 + \frac{t_0}{R_0^2}E_0) \leq C_0\epsilon_1^2 E_0.
\end{split}\end{equation}

Now we are going to estimate $\Ee_{R_0}(t)$. For any $x\in\overline{\Omega}$, let cut-off function $\phi\in C_0^\infty(B_{2R_0}(x))$ such that $0\leq\phi\leq 1,\phi\equiv 1$ on $B_{R_0}(x)$, and $|\nabla \phi|\lesssim 1/R_0$. Then by interior and boundary energy inequalities  \eqref{local_energy}, we have
\begin{equation}\label{est: Ee}\begin{split}
    & \sup_{0\leq t\leq t_0} \int_{\Omega\cap B_{R_0}(x)} (|u|^2 + |\nabla d|^2) - \Ee_{2R_0}(0) \leq \sup_{0\leq t\leq t_0} \int_{\Omega\cap B_{2R_0}(x)} (|u|^2+|\nabla d|^2)\phi - \Ee_{2R_0}(0) \\
   &  \lesssim \int_{\Omega_{t_0}} |\nabla \phi| (|u|^3 + |P-P_\Omega||u| + |\nabla u||u| + |d_t||\nabla d|) \\
   & \lesssim \Big(\frac{t_0}{R_0^2}\Big)^{\frac14}\Big( \|u\|_{L^4(\Omega_{t_0})}^3 + \|\nabla u\|_{L^2(\Omega_{t_0})}\|u\|_{L^4(\Omega_{t_0})}  + \|\nabla d\|^2_{L^4(\Omega_{t_0})}\|u\|_{L^4(\Omega_{t_0})} + \|d_t\|_{L^2(\Omega_{t_0})}\|\nabla d\|_{L^4(\Omega_{t_0})} \Big) \\ & + \|P-P_\Omega\|_{L^{4,\frac43}(\Omega_{t_0})}\|u\|_{L^4(B_{2R_0}(x)\times(0,t_0)\cap \Omega_{t_0})}.
\end{split}\end{equation}
Notice that $\|u\|_{L^2(\Omega_{t_0})}\leq (t_0 E_0)^{\frac12}\leq E_0^{\frac12}$. To estimate $d_t$, we test $\eqref{LCF}_3$ with $d_t$ and use \eqref{est: L4-gradd}-\eqref{est: L4-u} to obtain
\[ \int_{\Omega_{t_0}} |d_t|^2 \lesssim \int_{\Omega} |\nabla d_0|^2 + \int_{\Omega_{t_0}} |u|^2|\nabla d|^2 \lesssim E_0 + \|u\|^2_{L^4(\Omega_{t_0})}\|\nabla d\|^2_{L^4(\Omega_{t_0})}  \leq C_0 E_0. \]
The above inequality, together with Lemma \ref{lem: est pressure}, \eqref{est: L4-gradd}, \eqref{est: L4-u}, into \eqref{est: Ee}, yields
\[ 2\epsilon_1^2 = \sup_{0\leq t\leq t_0} \int_{\Omega\cap B_{R_0}(x)} (|u|^2 + |\nabla d|^2) 
\leq \epsilon_1^2 + C_0\Big(\frac{t_0}{R_0^2}\Big)^{\frac14} \epsilon_1^{\frac12}E_0^{\frac34} + \mathcal{O}(R_0 ) \epsilon_1^{\frac12}E_0^{\frac34} ,\]
where term $\|u\|_{L^4(B_{R_0}(x))\times(0,t_0)\cap \Omega_{t_0}}$ in \eqref{est: Ee} gives higher order term of $R_0$ which becomes negligible if we take $R_0\leq \epsilon_1^2$.
This implies \[ t_0\geq \frac{\epsilon_1^6}{C_0^4 E_0^3} R_0^2 := \theta_0R_0^2.\]
The proof is thus complete by taking $T_0=t_0$.
\end{proof}

\subsection*{Proof of Theorem \ref{Thm B}}
\begin{proof}
We follow the argument in \cite{LLW10}.

\emph{Step 1: Approximation of initial data}.
    Because we only assume  initial condition $u_0\in\mathbf{L}$ and $d_0\in\mathbf{J}$, we need to approximate them by smooth functions  so that we can utilize Lemma  \ref{lem: life span}. By density of smooth maps   $ L^2(\Omega,\RR^2) $ and by Lemma \ref{lem: density_Sobolev}, we assume that there exist $\{u_0^k\}_{k=1}^\infty\subseteq C^\infty(\Omega,\RR^2)\cap C^{2,\alpha}(\overline{\Omega},\RR^2)\cap \mathbf{L}$  and $\{d_0^k\}\subseteq C^\infty(\Omega,\SS^2)\cap C^{2,\alpha}(\overline{\Omega},\SS^2)\cap \mathbf{J}$ such that
    \begin{align*}
       \lim_{k\to\infty} \|u_0^k-u_0\|_{L^2(\Omega)} & = 0, \\
       \lim_{k\to\infty} \| d_0^k - d_0\|_{H^1(\Omega)} & = 0 .
    \end{align*}
    By the absolute continuity of $\int |u_0|^2 + |\nabla d_0|^2$, there exists $R_0>0$ such that 
    \[ \sup_{x\in\overline{\Omega}} \int_{\Omega\cap B_{2R_0}(x)} |u_0|^2 + |\nabla d_0|^2 \leq \frac{\epsilon_1^2}{2}, \]
    where $\epsilon_1>0$ is given in Lemma \ref{lem: life span}. By the strong convergence of $(u_0^k,\nabla d_0^k)\to(u_0,\nabla d_0)$ in $L^2(\Omega)$, it follows that 
    \[\sup_{x\in\overline{\Omega}} \int_{\Omega\cap B_{2R_0}(x)} |u_0^k|^2 + |\nabla d_0^k|^2 \leq  \epsilon_1^2 \ \text{ for } k  \text{ sufficiently large}. \]
    By discarding finitely many $k$'s and taking the rest as subsequence, we may assume without loss of generality that the above holds for $k\geq 1$. As a consequence, Lemma \ref{lem: life span} states that there exists $\theta_0=(\epsilon_1,E_0)\in(0,1)$ and life-spans $T_0^k\geq \theta_0 R_0^2$ such that each initial condition $(u_0^k,d_0^k)$ admits a  $(u^k,d^k)\in C^\infty(\Omega\times(0,T_0^k),\RR^2\times\SS^2)\cap C^{2,1}_\alpha(\overline{\Omega}\times(0,T_0^k),\RR^2\times\SS^2)$ as solution to \eqref{LCF}-\eqref{assumption1}. 
    
    Moreover, we have
    \begin{equation}\label{eqn: ukdk life_span} \sup_{(x,t)\in \overline{\Omega}_{T_0^k}} \int_{\Omega\cap B_{R_0}(x)} (|u^k|^2 + |\nabla d^k|^2) (\cdot,t) \leq 2\epsilon_1^2, \end{equation}
    and Lemma \ref{lem: global energy} gives energy estimate
    \begin{equation}\label{eqn: ukdk energy} \sup_{0\leq t\leq T_0^k} \int_\Omega (|u^k|^2 + |\nabla d^k|^2)(\cdot,t) + \int_{\Omega_{T_0^k}} (4|\DD u^k|^2 + 2|\Delta d^k + |
    \nabla d^k|^2 d^k|^2) \lesssim \int_\Omega |u_0^k|^2 + |\nabla d_0^k|^2 \lesssim E_0. \end{equation}
    The above two estimates together with Lemma \ref{Ladyzhenskaya} imply that
    \begin{align}\label{eqn: ukdkL4}
        \int_{\Omega_{T_0^k}} |u^k|^4 + |\nabla d^k|^4 & \lesssim \epsilon_1^2 E_0,   \\
        \label{eqn: dk_tD_L2}
         \int_{\Omega_{T_0^k}} |d_t^k|^2 + |\nabla^2d^k|^2 & \lesssim E_0. 
    \end{align} 
    We use Lemma \ref{lem: est pressure} and estimates \eqref{eqn: ukdk life_span}, \eqref{eqn: ukdkL4} and \eqref{eqn: dk_tD_L2} to conclude that
    \begin{equation}\label{eqn: dPkL43}
    \|\nabla P^k\|_{L^{\frac43}(\Omega_{T_0^k})} \lesssim \| \nabla u^k\|_{L^2(\Omega_{T_0^k})} \|u^k\|_{L^4(\Omega_{T_0^k})} + \|\nabla^2d^k\|_{L^2(\Omega_{T_0^k})}\|\nabla d^k\|_{L^4(\Omega_{T_0^k})} \lesssim \epsilon_1^{\frac12} E_0^{\frac34}.
    \end{equation}
    We collect estimates \eqref{eqn: ukdkL4}, \eqref{eqn: dk_tD_L2}, \eqref{eqn: dPkL43} and apply Theorem \ref{Thm A} to obtain that for any $\delta>0$,
    \begin{equation}
        \begin{split}
            \|(u^k,d^k)\|_{C^{2,1}_\alpha(\overline{\Omega}\times[\delta,T_0^k])} & \leq C\left(\delta,E_0, \|u^k\|_{L^4(\Omega_{T_0^k})}, \|\nabla d^k\|_{L^4(\Omega_{T_0^k})}, \|\nabla P^k\|_{L^{\frac43}(\Omega_{T_0^k})} \right) \\
            & \leq C(\delta,E_0, \epsilon_1).
        \end{split}
    \end{equation}
    Furthermore, for any compact subdomain $K\Subset \Omega$ and $\delta>0$,
    \[ \|(u^k,d^k)\|_{C^l(K\times[\delta, T_0^k])} \leq C(\dist(K,\partial\Omega), \delta, l, E_0). \]
    Hence, after possibly passing to subsequences, there exist $T_0\geq \theta_0R^2$, $u\in W^{1,0}_2(\Omega_{T_0},\RR^2)$, and $d\in W^{2,1}_2(\Omega_{T_0},\SS^2)$ such that 
    \begin{align*}
       & u^k \rightharpoonup u \text{ weakly in } W^{1,0}_2(\Omega_{T_0},\RR^2), \ d^k\rightharpoonup d \text{ weakly in } W^{2,1}_2(\Omega_{T_0},\SS^2), \\
      &  \lim_{k\to\infty} \|u^k-u\|_{L^4(\Omega_{T_0})} = 0, \\
      &  \lim_{k\to\infty} \|d^k-d\|_{L^4(\Omega_{T_0})}  + \|\nabla d^k-\nabla d\|_{L^2(\Omega_{T_0})} = 0,
    \end{align*}
    and for any $l\geq 2, \delta>0, \gamma<\beta$, and compact $K\Subset \Omega$,
    \begin{align*}
       &  \lim_{k\to\infty} \| (u^k,d^k)- (u,d) \|_{C^l(K\times[\delta,T_0])} = 0, \\
       & \lim_{k\to\infty} \| (u^k,d^k)-(u,d) \|_{C^{2,1}_\gamma (\overline{\Omega}\times [\delta,T_0])} = 0 .
    \end{align*}
    As a result, $(u,d)\in C^{\infty}(\Omega\times(0,T_0],\RR^2\times\SS^2)\cap C^{2,1}_\beta (\overline{\Omega}\times(0,T_0],\RR^2\times\SS^2 )$ solves \eqref{LCF}-\eqref{IC} in $\Omega\times(0,T_0]$ and satisfies the boundary condition. 
    
By $\eqref{LCF}_1$, $u^k_t\in L^2([0,T_0^k], H^{-1}(\Omega))$ and $\|u^k_t\|_{L^2([0,T_0^k], H^{-1}(\Omega))}\leq CE_0$, this together with \eqref{eqn: dk_tD_L2} states that after possibly passing to a subsequence,
    $(u,\nabla d)(\cdot,t)\rightharpoonup
    (u_0,\nabla d_0)$ weakly in $L^2(\Omega)$ as $t\to 0$. Consequently, \[ E(0) \leq \liminf_{t\to 0} E(t).\]
    On the other hand, steps in energy estimates \eqref{eqn: ukdk energy} give
    \[ E(0) \geq \limsup_{t\to 0} E(t). \]
    This implies that $(u,\nabla d)(\cdot,t)$ converges to $(u_0,\nabla d_0)$ strongly in $L^2(\Omega)$. Hence, $(u,d)$ satisfies initial condition \eqref{IC}.
    
\vspace{5mm}
\emph{Step 2: weak extension beyond singular time.}
     Let $T_1\in(T_0,\infty)$ be the first singular time of $(u,d)$, i.e.,
     \[ (u,d)\in C^\infty(\Omega\times(0,T_1),\RR^2\times\SS^2)\cap C^{2,1}_\beta (\overline{\Omega}\times(0,T_1),\RR^2\times\SS^2  ), \]
     but 
     \[ (u,d) \not\in C^\infty(\Omega\times(0,T_1],\RR^2\times\SS^2)\cap C^{2,1}_\beta (\overline{\Omega}\times(0,T_1],\RR^2\times\SS^2  ) \]
     Now, we would like to extend this weak solution in time.  To do so, we shall investigate and define new ``initial'' data at $t=T_1$.

    We claim that $(u,d)\in C^0([0,T_1], L^2(\Omega))$. Indeed, we test $\eqref{LCF}_3$ with $\phi\in H_0^2(\Omega,\RR^3)$ and obtain 
    \begin{align*} |\langle d_t,\phi\rangle| & \lesssim \|\nabla d\|_{L^2(\Omega)} \|\nabla \phi\|_{L^2(\Omega)} + (\|u\|_{L^2(\Omega)}\|\nabla d\|_{L^2(\Omega)} + \|\nabla d\|^2_{L^2(\Omega)} ) \|\phi\|_{C^0(\Omega)}  \\
    & \lesssim (\|\nabla d\|_{L^2(\Omega)}  + \|u\|_{L^2(\Omega)}\|\nabla d\|_{L^2(\Omega)} + \|\nabla d\|^2_{L^2(\Omega)} ) \|\phi\|_{H^2(\Omega)},
    \end{align*}
where we have used $\|\phi\|_{C^0(\Omega)}\lesssim \|\phi\|_{H^2(\Omega)}$ by the fact that $H^2_0(\Omega)\subseteq C^0(\Omega)$. Therefore, $d_t\in L^2([0,T_1], H^{-2}(\Omega))$, which together with the fact that $d\in L^2([0,T_1], H^1(\Omega))$, implies that $d\in C^0([0,T_1], L^2(\Omega))$. Similarly, we test $\eqref{LCF}_1$ with $\phi\in H^3_0(\Omega,\RR^2)$ and $\nabla\cdot \phi=0$ to obtain
\begin{align*}
    |\langle u_t,\phi\rangle| & \lesssim \|\nabla u\|_{L^2(\Omega)} \|\nabla\phi\|_{L^2(\Omega)} + \|u\|_{L^2(\Omega)}\|\nabla u\|_{L^2(\Omega)} \|\phi\|_{C^0(\Omega)} + \|\nabla u\|_{L^2(\Omega)}^2 \|\nabla \phi\|_{C^0(\Omega)} \\
    & \lesssim (\|\nabla u\|_{L^2(\Omega)} + \|u\|_{L^2(\Omega)}\|\nabla u\|_{L^2(\Omega)} + \|\nabla u\|^2_{L^2(\Omega)} )\|\phi\|_{H^3(\Omega)},
\end{align*}
where we have used $\|\phi\|_{C^1(\Omega)}\lesssim \|\phi\|_{H^3(\Omega)}$ by the fact that $H^3_0(\Omega)\subseteq C^1(\Omega)$. Since $\nabla\cdot u_t=0$, it follows that $u_t\in L^2([0,T_1], H^{-3}(\Omega))$. This, together with the fact that $u\in L^2([0,T_1], H^1(\Omega))$, implies that $u\in C^0([0,T_1], L^2(\Omega))$, and thus the claim is valid.

Now  $(u,d)\in C^0([0,T_1], L^2(\Omega))$  means that we can define 
\[ (u(T_1), d(T_1)) = \lim_{t\nearrow T_1} (u(t), d(t)) \text{ in } L^2(\Omega). \]
Then the energy estimate in Lemma \ref{lem: global energy} yields $\nabla d\in L^\infty([0,T_1], L^2(\Omega))$. Thus, we have the weak convergence $\nabla d(t)\rightharpoonup \nabla d(T_1)$ in $L^2(\Omega)$. In particular, $u(T_1) \in \Lbf$  and $d(T_1)\in H^1(\Omega) $. Moreover, $u(T_1)$ and $d(T_1)$ satisfy the partially free boundary condition  \eqref{PFB} since $u(t)$ and $d(t)$ satisfy it. 

Now we can use $(u(T_1), d(T_1))$ as initial data in the above procedure to obtain a weak extension of $(u,d)$  beyond $T_1$ that solves \eqref{LCF}-\eqref{assumption1}. We may be confronted with another singular time and we continue to process a weak extension in this scheme. We would like to  show that there can be at most finitely many such singular times, and we can construct an external weak solution beyond afterwards.

Notice that in the  study of heat flow of harmonic maps, each singularity carries a loss of energy. Here we will prove a similar result: at each singular time, the energy is lost by at least $\epsilon_1^2$. By definition, $T_1$ is the first singular time of $(u,d)$, then Lemma \ref{lem: life span} states that we must have
     \[ \limsup_{t\nearrow T_1} \max_{x\in \overline{\Omega}} \int_{\Omega\cap B_{2R_0(x)}} (|u|^2 + |\nabla d|^2 )(\cdot,t) > \epsilon_1^2 .\]
     This means that there exists $t_i\nearrow T_1$ and $x_0\in \overline{\Omega}$ such that for any $R>0$,
     \[\limsup_{t_i\nearrow T_1}  \int_{\Omega\cap B_{2R_0(x_0)}} (|u|^2 + |\nabla d|^2 )(\cdot,t_i) > \epsilon_1^2 . \]
    And we observe that
    \begin{align*}
        & \int_\Omega (|u|^2 + |\nabla d|^2)(\cdot,T_1) = \lim_{R\to 0} \int_{\Omega\setminus B_R(x_0)} (|u|^2 + |\nabla d|^2)(\cdot,T_1) \\
        & \leq \lim_{R\to 0}\liminf_{t_i\nearrow T_1} \int_{\Omega\setminus B_R(x_0)} (|u|^2 + |\nabla d|^2)(\cdot,t_i) \\
        & \leq 
        \lim_{R\to 0}\liminf_{t_i\nearrow T_1} \int_{\Omega} (|u|^2 + |\nabla d|^2)(\cdot,t_i) - \lim_{R\to 0}\limsup_{t_i\nearrow T_1} \int_{\Omega\cap B_R(x_0)} (|u|^2 + |\nabla d|^2)(\cdot,t_i) \\
        & \leq \liminf_{t_i\nearrow T_1} \int_\Omega (|u|^2 + |\nabla d|^2)(\cdot,t_i) - \epsilon_1^2\leq E_0-\epsilon_1^2.
    \end{align*}
     This shows that each singular time   takes away energy at least $\epsilon_1^2$, so the number of singular time is bounded by $E_0/\epsilon_1^2$. Let $0<T_L<\infty$ be the last singular time, then we must have
     \[ E(T_L) = \int_\Omega (|u|^2 + |\nabla d|^2)(\cdot,T_L) < \epsilon_1^2. \]
     Consequently, if we take $(u(T_L), d(T_L))$ as initial condition and construct a weak solution to \eqref{LCF}-\eqref{IC}, this weak extension will be an eternal weak solution.

    \vspace{5mm}
    \emph{Step 3: Uniqueness of global weak solution}.  
Uniqueness of regular solution to $\eqref{LCF}$ with Dirichlet boundary condition was shown in \cite{Lin-Wang10}, and their techniques involve proving $$\sqrt{t}(\|u(t)\|_{L^\infty(\Omega)} + \|\nabla d(t)\|_{L^\infty(\Omega)})\to 0$$
 as $t\to0$, the $L^p$--$L^q$ regularity of the heat and Stokes operators. We shall follow the same argument as in \cite{Lin-Wang10}.  Suppose that there are two weak solutions $(u_i,d_i),i=1,2$ to $\eqref{LCF}$-$\eqref{IC}$ with $\eqref{assumption1}$ such that $u_i\in L^\infty([0,T], L^2(\Omega,\RR^2))\cap L^2([0,T], H^1(\Omega,\RR^2))$ and $  d_i\in L^\infty([0,T], H^1(\Omega,\SS^2))\cap L^2([0,T], H^2(\Omega,\SS^2))$. 

First we claim that $A_i(t)= \sup_{0<s<t} \sqrt{s}(\|u_i(s)\|_{L^\infty(\Omega)}  + \|\nabla d_i(s)\|_{L^\infty(\Omega)})<+\infty$. For $(x_0,t_0)\in\Omega\times(0,T)$, we  take $0<\tau \leq \sqrt{t_0}$.  Since $(u,d)$ solves $\eqref{LCF}$ on $\Omega\times[0,T]$, by a scaling argument, taking $(v,g)(y,s)= (\tau u,d)(x+\tau y, \tau^2 + \tau^2 s)$ for $(y,s)\in P'= [-1,0]\times \frac{1}{\tau}\Omega$, one has $(v,g)$ solves $\eqref{LCF}$ on $P'$.  Thus one can apply Theorem \ref{Thm A} to conclude that $(v,g)\in C^\infty(P')\cap C^{2,1}_\alpha(\overline{P'})$.  This means $\|v\|_{L^\infty(P')} + \|\nabla g\|_{L^\infty(P')} \leq C(u,d) < +\infty$,  where $C(u,d)$ is independent of the scaling size  $\tau$. Back to the original scales, we obtain 
\[ \sup_{0<\tau\leq \sqrt{t_0}}\tau (\|u(\tau^2)\|_{L^\infty(\Omega)} + \|\nabla d(\tau^2)\|_{L^\infty(\Omega)}) < +\infty . \]

Next, we claim that $A_i(t)= \sup_{0<s<t} \sqrt{s}(\|u_i(s)\|_{L^\infty(\Omega)}  + \|\nabla d_i(s)\|_{L^\infty(\Omega)}) \to 0$ as $t\to 0$. 
From the global energy equality in Lemma \ref{lem: global energy}, we see that the energy $E_i(t)=\int_\Omega |u_i(t)|^2 + |\nabla d_i(t)|^2 $ is monotone decreasing with respect to $t\geq 0$. As a consequence, 
\[ \lim_{t\to 0} E_i(t) \leq E(0) = \int_\Omega |u_0|^2 + |\nabla d_0|^2. \]
On the other hand, since $(u_i(t), \nabla d_i(t))$ converges weakly to $(u_0,\nabla d_0)$ in $L^2(\Omega)$ as $t\to 0$, lower semi-continuity also gives $\lim_{t\to 0}E_i(t)\geq E_0$, and hence 
\[ \lim_{t\to 0} E_i(t) = E_0. \] So we have $E_i(t)\in C([0,T])$.  
It then follows that 
\[ \lim_{t\to 0} \sup_{x\in\overline{\Omega}} \int_{B_t(x)\times [0,t^2]} \sum_{i=1}^2 (|u_i|^2 + |\nabla d_i|^2) = 0, \]
and by the smoothness of $u_i$ and $\nabla d_i$, we have
\[ \lim_{t\to 0}\sup_{0<s<t} \sqrt{s}(\|u_i(s)\|_{L^\infty(\Omega)} + \|\nabla d_i(s)\|_{L^\infty(\Omega)} ) = 0. \]

Finally, we claim that $\lim_{t\to0}A_i(t)=0$ for $i=1,2$ implies that $(u_1,d_1)\equiv (u_2,d_2)$ on some $\Omega\times[0,t_0]$.  Let $\tilde{u}=u_1-u_2$ and $\tilde{d}=d_1-d_2$, then we take the difference of the corresponding equations and obtain
\begin{equation}\label{eqn: difference}
    \begin{split}
        \partial_t \tilde{u} - \bbA \tilde{u} & = - \PP \nabla\cdot(\tilde{u}\otimes u_1 + u_2\otimes \tilde{u} + \nabla \tilde{d}\otimes \nabla d_1 +\nabla d_2\otimes \nabla \tilde{d} ) \\
        \nabla\cdot \tilde{u} & = 0 \\
        \partial_t \tilde{d} - \Delta \tilde{d} & = [(\nabla d_1 + \nabla d_2)\cdot \nabla \tilde{d} d_1 + |\nabla d_2|^2 \tilde{d}] - [\tilde{u}\cdot\nabla d_1 + u_2\cdot\nabla \tilde{d}] \\
        (\tilde{u},\tilde{d})|_{t=0} & = 0\\
        (\tilde{u},\tilde{d}) & \text{ satisfies } \eqref{PFB}.
    \end{split}
\end{equation}

Now for $\delta\in(0,1)$ we define
\[  D_\delta(t) = t^{\frac{1-\delta}{2}} (\|u_1(t)\|_{L^{2/\delta}(\Omega)} +  \|u_2(t)\|_{L^{2/\delta}(\Omega)} + \|\nabla d_1(t)\|_{L^{2/\delta}(\Omega)} + \|\nabla d_2(t)\|_{L^{2/\delta}(\Omega)}  ) .\]
By the interpolation inequality, we obtain
\[ D_\delta(t) \leq (E_1(t) + E_2(t))^\delta (A_1(t) + A_2(t))^{1-\delta}. \]
Recall that by Duhamel's formula, we have
\begin{equation*}
    \begin{split}
        \tilde{u}(t) & = -\int_{0}^t e^{-(t-s)\bbA} \PP \nabla \cdot (\tilde{u}\otimes u_1 + u_2\otimes \tilde{u} + \nabla\tilde{d} \otimes \nabla d_1 + \nabla d_2 \otimes \nabla \tilde{d}  )(s), \\
        \tilde{d}(t) & =\int_0^t e^{-(t-s)\Delta} [(\nabla d_1 + \nabla d_2)\cdot \nabla \tilde{d} d_1 + |\nabla d_2|^2\tilde{d} - \tilde{u}\cdot\nabla d_1 - u_2\cdot\nabla\tilde{d} ](s).
    \end{split}
\end{equation*}
Thus we apply Lemma \ref{lem: LpLq reg Heat} with $q=2/\delta$ and $p=1$, together with H\"older inequality, and obtain 

 \begin{equation*}
     \begin{split}
         \|\tilde{d}(t)\|_{L^{2/\delta}(\Omega)} & \lesssim \int_0^t (t-s)^{-\frac{2-\delta}{2}} \bigg[ \sum_{i=1}^2 \|\nabla d_i(s)\|_{L^2(\Omega)} + \|u_i(s)\|_{L^2(\Omega)} \bigg]^2 ds   \\
         & \lesssim \bigg( \int_0^t (t-s)^{-\frac{2-\delta}{2}}ds \bigg) \sup_{0\leq s\leq t} (E_1+E_2)(s) \lesssim t^{\delta/2} \sup_{0\leq s\leq t} (E_1+E_2)(s).
     \end{split}
 \end{equation*}
Moreover, we can apply Lemma \ref{lem: LpLq reg Heat} with $q=2/\delta$ and $p=2/(\delta + 1)$, together with the H\"older inequality, to obtain
\begin{equation}\label{eqn: til_d_est}
    \begin{split}
       &  \| \tilde{d} (t) \|_{L^{2/\delta} (\Omega)}
       \\
       & \lesssim \int_0^t (t-s)^{-1/2} \Big[ \sum_{i=1}^2 \|\nabla d_i(s) \|_{L^{2/\delta}(\Omega)}  + \| u_i(s) \|_{L^{2/\delta}(\Omega)} \Big]\Big( \|\nabla \tilde{d}(s)\|_{L^2(\Omega)} + \|\tilde{u}\|_{L^2(\Omega)} \Big) 
        \\
        & + \int_{0}^t  (t-s)^{-1/2} \|\nabla d_2\|_{L^\infty(\Omega)}   \|\nabla d_2\|_{L^2(\Omega)} \|\tilde{d}(s)\|_{L^{2/\delta}(\Omega)} 
        \\
        & \lesssim \Big(  \int_0^t (t-s)^{-1/2}s^{(\delta-1)/2}  ds \Big) \Big(  \sup_{0<s\leq t} D_\delta(s)   \Big) \Big(  \sup_{0\leq s\leq t } \|\nabla \tilde{d} (s)\|_{L^2(\Omega)} + \|\tilde{u}(s)\|_{L^2(\Omega)}   \Big)   \\
        & + \Big(  \int_0^t (t-s)^{-1/2}s^{(\delta-1)/2}  ds \Big)   \Big(  \sup_{0<s\leq t} A_2(s)   \Big) \Big(  \sup_{0\leq s\leq t} (E_1+E_2)(s)  \Big) \Big(  \sup_{0<s\leq t}  s^{-\delta/2} \|\tilde{d}(s)\|_{L^{2/\delta}(\Omega)}  \Big) 
        \\
        & \lesssim t^{\delta/2} \Big[ \sup_{0<s\leq t} D_\delta(s)  + \big( \sup_{0<s\leq t} A_2(s) \big) \big( \sup_{0<s\leq t} (E_1+E_2)(s)  \big)  \Big] 
        \\
        &\quad \times \sup_{0<s\leq t}\Big[  \|\nabla \tilde{d}(s)\|_{L^2(\Omega)}  + \|\tilde{u}(s)\|_{L^2(\Omega)} + s^{-\delta/2} \|\tilde{d}(s)\|_{L^{2/\delta}(\Omega)}  \Big],
    \end{split}
\end{equation}
which converges to $0$ as $t\to 0$.
Also, we apply Lemma \ref{lem: LpLq reg Heat} with $q=2$ and $p=2/(\delta+1)$, together with H\"older inequality, to obtain
\begin{equation}\label{eqn: til_Dd_est}
    \begin{split}
        & \|\nabla \tilde{d} (t)\|_{L^{2/\delta}(\Omega)}
        \\
        & \lesssim \int_0^t (t-s)^{-(1+\delta)/2} \Big(  \|\nabla d_1(s)\|_{L^{2/\delta}(\Omega)} + \|\nabla d_2(s)\|_{L^{2/\delta}(\Omega)} \Big)  \|\nabla \tilde{d}(s)\|_{L^2(\Omega)} ds  \\
        & + \int_0^t (t-s)^{-(1+\delta)/2}  \Big(  \|\nabla d_1(s)\|_{L^{2/\delta}(\Omega)} \|\tilde{u}(s)\|_{L^2(\Omega)} + \|u_2(s)\|_{L^{2/\delta}(\Omega)} \|\nabla \tilde{d}(s)\|_{L^2(\Omega)}  \Big) ds 
        \\
        & + \int_0^t (t-s)^{-(1+\delta)/2} \|\nabla d_2(s)\|_{L^\infty(\Omega)} \|\nabla d_2(s)\|_{L^2(\Omega)} \|\tilde{d}(s)\|_{L^{2/\delta}(\Omega)} ds 
        \\
        & \lesssim  \Big(  \int_0^t (t-s)^{-(1+\delta)/2} s^{(\delta-1)/2} ds \Big) \Big(  \sup_{0<s\leq t} D_\delta(s)  \Big) \Big( \sup_{0\leq s\leq t} (\|\tilde{u}(s)\|_{L^2(\Omega)} + \|\nabla \tilde{d}(s)\|_{L^2(\Omega)})  \Big) \\
        & +  \Big(  \int_0^t (t-s)^{-(1+\delta)/2} s^{(\delta-1)/2} 
        ds \Big) \Big(  \sup_{0<s\leq t} A_2(s) \Big) \Big( \sup_{0\leq s\leq t} (E_1+E_2)(s) \Big)\Big(  \sup_{0<s\leq t} s^{-\delta/2}\|\tilde{d}(s)\|_{L^{2/\delta}(\Omega)} \Big)
        \\
        &\leq  C\Big(  \sup_{0<s\leq t} D_\delta(s)  \Big) \Big( \sup_{0\leq s\leq t} (\|\tilde{u}(s)\|_{L^2(\Omega)} + \|\nabla \tilde{d}(s)\|_{L^2(\Omega)})  \Big) 
        \\
        & + C\Big(  \sup_{0<s\leq t} A_2(s) \Big) \Big( \sup_{0\leq s\leq t} (E_1+E_2)(s) \Big)\Big(  \sup_{0<s\leq t} s^{-\delta/2}\|\tilde{d}(s)\|_{L^{2/\delta}(\Omega)} \Big),
    \end{split}
\end{equation}
where we have used the fact that $\int_0^t (t-s)^{-(\delta+1)/2}s^{(\delta-1)/2} ds = \int_0^1 (1-s)^{-(1+\delta)/2}s^{(\delta-1)/2} < +\infty$, and likewise we conclude that it converges to 0 as $t\to 0$.

Now we apply Lemma \ref{lem: LpLq reg Stokes} with $2>q=\frac{2}{1+\delta}\frac{2}{2-\delta}> p = \frac{2}{1+\delta}$, $\gamma=1-\frac{p}{q}=\frac{\delta}{2}$, and obtain
\begin{equation}\label{eqn: til_u_est}
    \begin{split}
       & \|\tilde{u}\|_{L^{\frac{2}{1+\delta}\frac{2}{2-\delta}}(\Omega)}  
       \\
       & \lesssim \int_0^t (t-s)^{-(\delta+1)/2} \Big( \|u_1(s)\|_{L^{2/\delta}(\Omega)} + \|u_2(s)\|_{L^{2/\delta}(\Omega)} \Big)\|\tilde{u}(s)\|_{L^2(\Omega)} ds 
       \\
       & + \int_0^t (t-s)^{-(\delta+1)/2}  \Big( \|\nabla d_1(s)\|_{L^{2/\delta}(\Omega)} + \|\nabla d_2(s)\|_{L^{2/\delta}(\Omega)} \Big)\|\nabla \tilde{d}(s)\|_{L^2(\Omega)} ds 
       \\
       & \lesssim \Big( \int_0^t  (t-s)^{-(\delta+1)/2}s^{(\delta-1)/2} ds\Big) \Big( \sup_{0<s\leq t} D_\delta(s) \Big)\Big( \sup_{0\leq s\leq t} (\|\tilde{u}(s)\|_{L^2(\Omega)} + \|\nabla \tilde{d}\|_{L^2(\Omega)} )  \Big)
       \\
       & \leq C\Big(  \sup_{0<s\leq t} D_\delta(s)  \Big) \Big( \sup_{0\leq s\leq t} (\|\tilde{u}(s)\|_{L^q(\Omega)} + \|\nabla \tilde{d}(s)\|_{L^2(\Omega)})  \Big) .
    \end{split}
\end{equation}
Finally, for $0<t\leq t_0$, we define
\[ \Phi(t) = \sup_{0<s\leq t} \Big( \|\nabla\tilde{d}(s)\|_{L^2(\Omega)} + s^{-\frac{\delta}{2}}\|\tilde{d}(s)\|_{L^{\frac{2}{\delta}}(\Omega)} + \|\tilde{u}(s)\|_{L^{\frac{2}{1+\delta}\frac{2}{2-\delta}}(\Omega)} \Big).   \]
From estimates \eqref{eqn: til_d_est}, \eqref{eqn: til_Dd_est}, \eqref{eqn: til_u_est}, and the H\"older inequality $\|\cdot\|_{L^2(\Omega)}\lesssim \|\cdot\|_{L^{\frac{2}{1+\delta} \frac{2}{2-\delta}}(\Omega)}$, we have
\[ \Phi(t) \leq C\Big[ \sup_{0<s\leq t} D_\delta(s)  + \Big(\sup_{0<s\leq t} (A_1+A_2)(s) \Big) \Big(\sup_{0\leq s\leq t} (E_1+E_2)(s) \Big)\Big]\Phi(t) \leq  \frac{1}{2}\Phi(t)\]
as long as $t_0>0$ is sufficiently small such that 
\[ C \Big[ \sup_{0<s\leq t} D_\delta(s)  + \Big(\sup_{0<s\leq t} (A_1+A_2)(s) \Big) \Big(\sup_{0\leq s\leq t} (E_1+E_2)(s) \Big)\Big] \leq C(\epsilon^{1-\delta} + \epsilon) \leq \frac{1}{2} . \]
As a consequence, $\Phi(t)\equiv 0$ on $(0,t_0]$ and hence $(u_1,d_1)\equiv (u_2,d_2)$ on $\Omega\times(0,t_0]$. Furthermore, this implies that our choice for weak extension beyond singular time is also unique.

    \vspace{5mm}
     \emph{Step 4: Blow-up analysis}. We have established (1), (2), (4), and the first half of (3) in Theorem \ref{Thm B}.  It remains to carry out the blow-up analysis at each singular time. There exists $0<t_0<T_1, t_m\nearrow T_1, r_m\searrow 0$ such that
     \begin{equation}\label{eqn: BU concentration}
         \epsilon_1^2 = \sup_{x\in\overline{\Omega}, t_0\leq t\leq t_m} \int_{\Omega\cap B_{r_m}(x)} (|u|^2 + |\nabla d|^2),
     \end{equation}
     and we use Lemma \ref{lem: life span} in the opposite way to obtain $\{x_m\}_{m=1}^\infty\subseteq \Omega$ such that
     \begin{equation}\label{eqn: BU initial} \int_{\Omega\cap B_{2r_m}(x_m)} (|u|^2 + |\nabla d|^2)(\cdot, t_m-\theta_0 r_m^2) \geq \frac12 \max_{x\in\overline{\Omega}} \int_{\Omega\cap B_{2r_m}(x)} (|u|^2 + |\nabla d|^2)(\cdot, t_m-\theta_0 r_m^2) \geq \frac12\epsilon_1^2. \end{equation}
     Energy estimate in Lemma \ref{lem: global energy}, \eqref{eqn: BU concentration}, and the Ladyzhenskaya's inequality state that
     \begin{equation}\label{eqn: BU global}
         \int_{\Omega\times[t_0,t_1]} (|u|^4 + |\nabla d|^4) \leq C(\epsilon_1, E_0).
     \end{equation}

     Denote $\Omega_m = r_m^{-1} (\Omega\setminus\{x_m\})$. Define the blow-up sequence $(u_m,d_m): \Omega_m\times[\frac{t_0-t_m}{r_m^2},0]$ by
     \[ u_m(x,t) = r_m u(x_m + r_m x, t_m + r_m^2t),\quad d_m(x,t)= d(x_m + r_mx, t_m + r_m^2t). \]
    It follows that $(u_m,d_m)$ solves \eqref{LCF} on $\Omega_m\times[\frac{t_0-t_m}{r_m^2},0]$, and \eqref{eqn: BU concentration} \eqref{eqn: BU initial} and \eqref{eqn: BU global} give us the following
    \begin{align*}
       &  \int_{\Omega_m\cap B_2(0)} (|u_m|^2 + |\nabla d_m|^2)(-\theta_0) \geq \frac12\epsilon_1^2 ,\\
       &  \int_{\Omega_m\cap B_1(x)} (|u_m|^2 + |\nabla d_m|^2)(t) \leq \epsilon_1^2, \  \forall x\in\Omega_m,\frac{t_0-t_m}{r_m^2}\leq t\leq 0,\\
       & \int_{\Omega_m\times[-\frac{t_0-t_m}{r_m^2},0]} |u_m|^4  + |\nabla d_m|^4 \leq C(\epsilon_1, E_0).
    \end{align*}

     By possibly passing to subsequences, we may assume without loss of generality that $x_m\to x_0\in\overline{\Omega}$ for some $x_0\in\overline{\Omega}$.

     \emph{Case 1: $x_0\in\Omega$}. Then we can assume $r_0<\dist(x_0,\partial\Omega)$ and $\Omega_m\to\RR^2$. Also we have $\frac{t_0-t_m}{r_0^2}\to-\infty$. Consequently, regularity result in Theorem \ref{Thm A} states that there exists a smooth solution $(u_\infty',d_\infty'): \RR^2\times (-\infty,0]\to\RR^2\times\SS^2$ such that it solves \eqref{LCF} and 
     \[ (u_m,d_m)\to (u_\infty',d_\infty') \ \text{ in } C^2_{{\rm loc}}(\RR^2\times [-\infty,0]).\]
     Because of the regularity of 2D Navier-Stokes equation and the phenomenon of separation of sphere in the heat flow of harmonic map, we would like to show that the singularity is attributed to $\nabla d$. First, we want to show that $u_\infty'\equiv0$. Indeed, take any parabolic cylinder $P_R\subseteq \RR^2\times[-\infty,0]$, since $u\in L^4(\Omega\times[0,T_1])$, we have 
     \[ \int_{P_R} |u_\infty'|^4 = \lim_{m\to\infty} \int_{P_R} |u_m|^4 = \lim_{m\to\infty} \int_{B_{Rr_m}(x_m)}\int_{[t_m-R^2r_m^2, t_m]} |u|^4 = 0.  \]
    Next, we claim that $d_\infty'$ is a nontrivial and smooth harmonic map with finite energy. In fact, since $\Delta d + |\nabla d|^2 d \in L^2(\Omega\times[0,T_1])$, we have, for any compact $K\subseteq \RR^2$, 
    \begin{align*} \int_{-2\theta_0}^0 \int_K |\Delta d_\infty' + |\nabla d_\infty'|^2 d_\infty'|^2 \leq \liminf_{m} \int_{-2\theta_0}^0\int_{\Omega_m} |\Delta d_m + |\nabla d_m|^2 d_m|^2 \\
    = \lim_{m\to\infty} \int_{t_m-2\theta_0r_m^2}^{t_m} \int_\Omega |\Delta d + |\nabla d|^2 d|^2 = 0 .
    \end{align*}
    This means $\partial_t d_\infty' + u\cdot\nabla d_\infty' =0$ on $\RR^2\times[-2\theta_0,0]$. Hence $\partial_t d_m=u_m=0$ and $d_\infty\in C^2(\RR^2,\SS^2)$ is a harmonic map. Also notice that
    \[ \int_{B_2 } |\nabla d_\infty'|^2 = \lim_{m\to\infty} \int_{B_2} (|u_m|^2 + |\nabla d_m|^2)(\cdot,-\theta_0) \geq \frac{\epsilon_1^2}{4},  \]
    and thus $d_\infty'$ is a nontrivial map. By the lower semi-continuity, for any $B_R\subseteq \RR^2$, 
    \[ \int_{B_R} |\nabla d_\infty'|^2 \leq \liminf_{m\to\infty} \int_{B_R} |\nabla d_m|^2(\cdot,-\theta_0) = \liminf_{m\to\infty} \int_{B_{r_mR(x_m)}} |\nabla d|^2(t_m-\theta_0r_m^2) \leq E_0 .\]
    This implies that $d_\infty'$ has finite energy. Studies of harmonic maps (see for instance \cite{Struwe85} and \cite{SU81}) show that $d_\infty'$ can be lifted to be a
    non-constant harmonic map from $\SS^2$ to $\SS^2$. In particular, the degree of $d_\infty'$ is nonzero and 
    \[ \int_{\RR^2} |\nabla d_\infty'|^2 \geq 8\pi|\deg(d_\infty')|\geq 8\pi .\]

    \emph{Case 2: $x_0\in\partial\Omega$.} If further $\lim_{m\to\infty} \frac{|x_m-x_0|}{r_m}=\infty$, then $\Omega_m\to\RR^2$. Then the same reasoning in Case 1 shows that $(u_m,d_m)\to(0,d_\infty')$ in $C^2_{{\rm loc}}(\RR^2)$, and $d_\infty'\in C^\infty(\RR^2,\SS^2)$ is a nontrivial harmonic map with finite energy. The other situation $\lim_{m\to\infty}\frac{|x_m-x_0|}{r_m}<\infty$ implies that $\Omega_m$ converges to a half plane and it will give  singularity at boundary: assume without loss of generality that $\frac{x_m-x_0}{r_m}\to (0,0)\in\RR^2$  and $\Omega_m\to \RR^2_+= \{(x_1,x_2): x_2> 0\} $. Because $d_m(x)= d(x_m+r_mx)$ for $x\in\partial\Omega_m$, we can show similarly that $(u_m,d_m)\to(0,d_\infty')$ in $C^2_{{\rm loc}}(\RR^2_+)$, where $d_\infty':\RR^2_+\to\SS^2$ is a nontrivial harmonic map with finite energy and $d_\infty'$ satisfies $\eqref{PFB}_2$ at $\partial\RR^2_+=\{(x_1,x_2): x_2=0\}$. The reflection given in Appendix \ref{appendix: reflection} allows us to use reflection to extend $d_\infty'$ to be a nontrivial harmonic map on the whole space $\RR^2$. Moreover, the reflection symmetry directly states that we have half energy
    \[ \int_{\RR^2_a} |\nabla d_\infty'|^2 = \frac12 \int_{\RR^2} |\nabla d_\infty'|^2 \geq 4\pi |\deg(d_\infty')| \geq 4\pi. \]
\end{proof}

\bigskip

\section{Eternal behavior: proof of Theorem \ref{Thm C} and Theorem \ref{Thm D}}\label{sec-7}

\begin{proof}[Proof of Theorem \ref{Thm C}]
     
    To prove (1), Lemma \ref{lem: global energy} states that there exists $t_k\to\infty$ such that for $(u_k,d_k)=(u(\cdot,t_k),d(\cdot,t_k))$
    \begin{align*}
       &  \int_\Omega |u_k|^2 + |\nabla d_k|^2 \leq E_0, \\
       &  \lim_{k\to\infty} \int_\Omega |\DD u_k|^2 + |\Delta d_k + |\nabla d_k|^2 d_k|^2 = 0.
    \end{align*}
    As a consequence, we have $\int_\Omega |\DD u_k|^2\to0$ as $k\to\infty$, and thus by weak limit there exists $t_k\to\infty$ such that $u_k\weakto u_\infty$ in $H^1(\Omega)$  such that $\DD u_\infty\equiv 0$.  In the situation of non-axisymmetric domain $\Omega$, we use the Korn's inequality in Lemma \ref{lem: Korn} to deduce that  $\nabla u_\infty\equiv0$, which together with Poincar\'e inequality for $H^1$-vector field with tangential boundary condition, i.e. $v\cdot\n = 0$ on $\partial\Omega$, implies that $u_\infty\equiv0$ in $H^1(\Omega)$. Meanwhile,  
     $\{d_k\}_{k=1}^\infty\subseteq H^1(\Omega,\SS^2)$ is a bounded sequence of approximated harmonic maps from $\Omega$ to $\SS^2$. Also, $\{d_k\}_{k=1}^\infty$ satisfies partially boundary condition $\eqref{PFB}_2$, and the tension field $\Delta d_k + |\nabla d_k|^2d_k$ converges to $0$ in $L^2(\Omega)$. By the energy identity result by Qing  \cite{Qing95} and Lin-Wang \cite{LW95},
we can conclude that there exists a harmonic map $d_\infty \in C^{2,\beta}(\Omega,\SS^2)$ with $d_\infty$ satisfying $\eqref{PFB}_2$, and there exist finitely many interior points $\{x_i\}_{i=1}^l\subseteq\Omega$ and boundary points $\{y_i\}_{i=1}^{l'}\subseteq\partial\Omega$ such that
\[ |\nabla d_k|^2 dx \rightharpoonup |\nabla d_\infty| dx + \sum_{i=1}^l 8\pi m_i \delta_{x_i} + \sum_{i=1}^l 4\pi m_i' \delta_{y_i} \]
for some subsequence $\{m_i\}_{i=1}^l,~\{m_i'\}_{i=1}^{l'}\subseteq \NN$. Note that this $d_\infty:\Omega\to\SS^2$ has different meaning from $d_\infty:\RR^2 (\text{or } \RR^2_a) \to \SS^2$ which is used in the blow-up analysis for Theorem \ref{Thm B}.

 In the end, it remains to prove (2). We first observe that the energy is not sufficient to evolve  finite time singularities. Suppose for the purpose of contradiction that  it blows up near the first singular time $T_1$. Then (3) in Theorem \ref{Thm B} implies that there exists a nontrivial harmonic map $w\in C^\infty(\RR^2,\SS^2)$ and 
\begin{equation}\label{eqn: near singular time} 8\pi\leq \int_{\RR^2}|\nabla w|^2 \leq 2\lim_{t\nearrow T_1} \int_\Omega (|u|^2 + |\nabla d|^2)(\cdot,t) \leq 2\int_\Omega |u_0|^2 + |\nabla d_0|^2 \leq 8\pi ,\end{equation}
where the factor 2 includes the possibility that we have singularity on the boundary. 
Thus there is no loss of energy, but then Lemma \ref{lem: global energy} implies that 
\begin{equation}\label{eqn: no singularity}
\int_0^{T_1}\int_\Omega |\DD u|^2 + |\Delta d + |\nabla d|^2d|^2 = 0, \end{equation}
and hence $\DD u=d_t\equiv0$ in $\Omega_{T_1}$. Therefore, $d(\cdot,t)=d_0\in C^{2,\beta}(\Omega,\SS^2),~ 0\leq t\leq T_1$, is a harmonic map, which contradicts the assumption that  $T_1$ is a singular
time.

Moreover, we would like to show that there is no blow-up of $\phi(t)=\max_{x\in\overline{\Omega},\tau\leq t} (|u|+ |\nabla d|)(x,\tau)$ at infinity. Suppose for the purpose of contradiction that there exists $t_k\to\infty$ and $x_k\in\overline{\Omega}$, such that
\[ \lambda_k = \phi(t_k) = (|u|+ |\nabla d|)(x_k,t_k) \to \infty.\]
Define $\Omega_k = \lambda_k(\Omega\setminus\{x_k\})$ and $(u_k,d_k):\Omega_k \times [-t_k\lambda_k^2,0]\to\RR^2\times\SS^2$ by
\[ u_k(x,t) = \frac1{\lambda_k} u\left(x_k+ \frac{x}{\lambda_k}, t_k + \frac{t}{\lambda_k^2}\right), \quad d_k(x,t) = d\left(x_k+\frac{x}{\lambda_k}, t_k + \frac{t}{\lambda_k^2}\right). \]
It follows that $(u_k,d_k)$ solves \eqref{LCF} on $\Omega_k\times[-t_k\lambda_k^2,0]$ and 
\[ 1= (|u_k|+ |\nabla d_k|)(0,0) \geq (|u_k|+ |\nabla d_k|)(x,t), \ \  \ \forall (x,t)\in \Omega_k\times[-t_k\lambda_k^2,0].  \]
With the same procedure as in Theorem \ref{Thm B} (3), we  conclude that there are two cases: either (i) $\Omega_k\to\RR^2$ and $(u_k,d_k)\to(0,w')$ in $C^2_{{\rm loc}}(\RR^2)$ where $w'\in C^\infty(\RR^2,\SS^2)$ is a nontrivial harmonic map with finite energy, or (ii) $\Omega_k\to\RR^2_a$ for some half plane $\RR_a:=\{a_1x_1+a_2x_2>a_0\}$ and $(u_k,d_k)\to(0,w')$ in $C^2_{{\rm loc}}(\RR^2_a)$ where $w': \RR_a^2 \to \SS^2$ is a nontrivial harmonic map with finite energy and satisfies $\eqref{PFB}_2$ at $\partial\RR_a^2$. Again, reflection symmetry in Appendix \ref{appendix: reflection} allows us to extend $w'$ to be a nontrivial harmonic map on the whole space $\RR^2$.

For case (i) and (ii), we perform the same scheme as in \eqref{eqn: near singular time}  and \eqref{eqn: no singularity} (with $T_1$ replaced by $\infty$), to conclude that
\[ \int_0^\infty \int_\Omega |\DD u|^2 + |\Delta d+ |\nabla d |^2d|^2 =0, \]
and the same reasoning yields $\DD u=d_t\equiv0$ on $\Omega\times[0,\infty)$. So $d(t)=d_0 \in C^{2,\beta}(\Omega,\SS^2), ~0\leq t<\infty$, is a harmonic map. This implies that $\phi(t)$ is constant and yields contradiction to the initial assumption that $\phi(t_k)\to\infty$.

We have shown that $\phi(t)$ is bounded on $t\in(0,\infty)$, and thus regularity results in Theorem \ref{Thm A} give that $\|u(\cdot,t)\|_{C^{2,\beta}}(\Omega),\|d(\cdot,t)\|_{C^{2,\beta}}(\Omega)$ stays bounded on $t\in(0,\infty)$. It follows that there exists a sequence $t_k\to\infty$ such that
\begin{align*}
    & \int_\Omega (|u|^2 + |\nabla d|^2)(x,t_k) \leq E_0 ,\\
    & \int_{\Omega} (|\DD u|^2 + |\Delta d + |\nabla d|^2d|)(x,t_k)\to 0 ,\\
    & \|u(\cdot,t_k)\|_{C^{2,\beta}(\Omega)} + \|d(\cdot,t_k)\|_{C^{2,\beta}}(\Omega) \leq C ,\\
    & (u(\cdot,t_k), d(\cdot,t_k)) \to (u_\infty,d_\infty) \ \text{ in } C^2(\overline{\Omega}, \RR^2\times\SS^2)
\end{align*}
for some $(u_\infty, d_\infty) \in C^{2,\beta}(\overline{\Omega},\RR^2) \times C^{2,\beta}(\overline{\Omega},\SS^2)$ satisfying partially free boundary condition $\eqref{PFB}$ on $\partial\Omega$.
\end{proof}

\medskip

\begin{proof}[Proof of Theorem \ref{Thm D}]
Since the domain $\Omega$ is axisymmetric, then it is either a disk or an annulus. From $\DD u_\infty\equiv0\in L^2(\Omega)$ one can see immediately that $u_\infty$ is a vortex flow of the form $u_\infty = c(x_2,-x_1)$, and it solves the  Navier Stokes equation $\eqref{LCF}_1$. Moreover we have $u_\infty\cdot\nabla d_\infty=0$. 

Case 1: If $c=0$, then  we have $u_\infty=0$ and $d_\infty$ is a harmonic map on $\Omega$.

Case 2: If $c\neq 0$ and $\Omega$ is a disk, since vortex flow $u_\infty \perp \hat{r}$ almost everywhere, then we know that each entry of $d_\infty$ is radial. We can show that $d$ is a constant harmonic map by ODE. Indeed, writing harmonic map equation in polar coordinates, we have
\[ \partial_{rr}d^i + \frac{1}{r}\partial_rd^i + |\nabla d|^2 d^i = 0, \]
and thus
\[  \sum_i \partial_{rr}d^i\partial_rd^i + \frac1r |\partial_r d^i|^2 + |\nabla d|^2 d^i\partial_rd^i = 0,
\]
where the last term $\sum_i d^i\partial_r d^i = 0$ because $|d|^2\equiv 1$. By taking $f(r)=|\nabla d|^2$, we have
\[ \partial_r f + \frac{2}{r}f=0, \]
and thus $f=\frac{\alpha^2}{r^2}$. We have harmonic map $d_\infty$ being smooth so that $f(0)<\infty$ and thus conclude that $f\equiv 0$.

Case 3: If $c\neq 0$ and $\Omega$ is an annulus. Then as above $f=|\nabla d|^2=\frac{\alpha^2}{r^2}$, but here $\alpha$ can be nonzero, so we have
\[ \partial_{rr}d^i + \frac{1}{r} \partial_r d^i + \frac{\alpha^2}{r^2}d^i = 0 . \]
We change the variable $r=e^s$. Then $\frac{d}{dr} = r^{-1} \frac{d}{ds}$ and $\frac{d^2}{dr^2}=r^{-2} (\frac{d}{ds}-1)\frac{d}{ds}$, which gives
\[  \partial_{ss}d^i + \alpha^2 d^i =0, \]
and hence
\[ d^i = A_i\sin(\alpha s +\phi_i ) = A_i \sin(\alpha \ln(r) + \phi_i )\]
for some constant $\phi_i$. And the partially free boundary conditions $\eqref{PFB}_2$ imply $\eqref{eqn: annulus_info}$.
\end{proof}

\bigskip

\appendix
\section{Boundary condition and the basic energy law}\label{appendix: reflection}

\medskip

We would like to show that the free boundary condition \eqref{PFB} is compatible with the basic energy law for the system \eqref{LCF}. 

We repeat the process in Lemma \ref{lem: global energy} and  obtain the basic energy law:
\begin{equation}\label{app: basic energy law}
    \frac12\frac{d}{dt}\Big(\int_\Omega |u|^2 + |\nabla d|^2  \Big) = - \int_\Omega \frac12|\nabla u + (\nabla u)^T|^2 -\int_\Omega \big| \Delta d+ |\nabla d|^2d \big|^2,
\end{equation}
which describes the property of  energy dissipation for the flow of liquid crystals.

Note that the system \eqref{LCF} has stress tensor 
\[S = \frac12(\nabla u + (\nabla u)^T) - P\Id_2 + \nabla d\odot\nabla d -\frac12|\nabla d|^2\Id_2,\]
so the physical compatibility condition requires $(S\cdot\n)_\tau =0$. Considering $\n\cdot\tau=0$ and the Navier perfect-slip boundary condition $\eqref{PFB}_2$, we have 
\[ 0 = \langle (\nabla d\odot \nabla d)\n,\tau  \rangle = \langle \nabla_\n d,\nabla_\tau d \rangle, \]
which gives the free boundary condition $\eqref{PFB}_2: \nabla_\n d\perp T_d\Sigma. $

In addition, it it worthy mentioning that in the case of half plane $\Omega=\RR^2_+$, the free boundary condition \eqref{PFB} gives a reflection across $\partial\RR^2_+$. First, free boundary condition \eqref{PFB} has a simple form in such case:
\[ \begin{cases}
    \pars u_1 = u_2 = 0 \\
    \pars d_1 = \pars d_2 = d_3 = 0
\end{cases} \ \text{ on } \partial\RR^2_+. \]
By performing even reflection for $u_1,d_1,d_2$ and odd reflection for $u_2,d_3$, we can use this reflection symmetry to extend our solution $(u,d)$ to the whole domain $\RR^2$. Explicitly,
\[ \tilde{u}(x_1,x_2,t)=\begin{bmatrix} \ \ u_1(x_1,-x_2,t)\\ -u_2(x_1,-x_2,t) \end{bmatrix}, \ \ \tilde{d}(x_1,x_2,t)=\begin{bmatrix}\ \ d_1(x_1,-x_2,t) \\ \ \ d_2(x_1,-x_2,t) \\ -d_3(x_1,-x_2,t) \end{bmatrix}, \ x_2<0  \]
It turns out that the partially free boundary condition \eqref{PFB} is automatically satisfied. We can further compute that
\[ \tilde{u}\cdot\nabla \tilde{u} = \begin{bmatrix}
  \ \  u_1 \parf u_1 +  u_2 \pars u_1  \\
    - (u_1 \parf u_2 + u_2\pars u_2) 
\end{bmatrix}(x_1,-x_2,t) ,  \]
\[ \tilde{u}\cdot\nabla \tilde{d}  = \begin{bmatrix} 
  \ \ u_1\parf d_1 + u_2\pars d_1 \\
  \ \ u_1\parf d_2 + u_2\pars d_2 \\
    -(u_1\parf d_3 + u_2\pars d_3)
\end{bmatrix}(x_1,-x_2,t),\]
and
\[\nabla\cdot(\nabla \tilde{d}\odot \nabla \tilde{d})=\begin{bmatrix}
  \ \ 2\parf d_k\parff d_k + \parss d_k\parf d_k + \parfs d_k\pars d_k \\
    -(2\pars d_k\parss d_k + \parff d_k\pars d_k + \parfs d_k\parf d_k)
\end{bmatrix}(x_1,-x_2,t).\]

Also observe that the partially free boundary condition \eqref{PFB} implies that $\nabla_{\n} P=0$ on $\partial\RR^2_+$. This follows from
\begin{align*} - \pars P & = \partial_t u_2 + u_1\parf u_2 + u_2\pars u_2 + \Delta u_2 + \pars d_k\parss d_k + \parff d_k\pars d_k \\
& = 0+ \parff u_2  - \parf(\pars u_1) + \pars d_3 \Delta d_3 \\
& = 0+ \pars d_3 (\partial_t d_3 + u_1\parf d_3+u_2\pars d_3 - |\nabla d|^2 d_3)  = 0.
\end{align*}
Hence, we perform even reflection for $P$ and get $\tilde{P}(x_1,x_2,t)=P(x_1,-x_2,t)$ for $x_2<0$. Then
we can then show that the structure of the system is preserved via the reflection $(u,d,P)\mapsto (\tilde{u},\tilde{d},\tilde{P})$, i.e.,
\[\begin{cases}
    \partial_t \tilde{u} + \tilde{u}\cdot\nabla\tilde{u} - \Delta\tilde{u} + \nabla\tilde{P} = -\nabla\cdot(\nabla\tilde{d} \odot \nabla\tilde{d}-\frac12|\nabla\tilde{d}|^2\Id_2) \\
    \nabla\cdot\tilde{u} = 0 \\
    \partial_t \tilde{d} + \tilde{u}\cdot\nabla\tilde{d} = \Delta\tilde{d}+|\nabla\tilde{d}|^2\tilde{d}.
\end{cases}\]
Such a reflection is not possible if we instead consider Navier no-slip boundary condition $u\equiv0$ on $\partial\Omega$, thus free boundary condition \eqref{PFB} is both physically meaningful and mathematically useful: it allows  us to convert boundary estimates to interior estimates and saves half of our labor. In particular, Lemma \ref{lem:boundary_smallness} follows directly from Lemma \ref{lem:local_smallness}. 

However, it is difficult to tackle the general situation with curved boundary $\partial\Omega$. We assume that $\partial\Omega$ is smooth, and we can flatten the boundary in a way such that in the new coordinate,  the velocity field is still divergence-free. One example in \cite{HS18} is to take a map $ \phi: \RR^2_+\cap B_r(z') \to\Omega\cap B_r(z)$   given by $\phi(x_1,x_2)=(x_1,x_2 + h(x_1))$, where $h(x_1)$ is locally the graph of the boundary $\partial\Omega$. Then we define the transformed vector field $v$ on $\RR^2_+$ by $v=Tu=u\circ \phi - (u\circ\phi)\cdot(h',0) e_2$. This transformation has the property that $\nabla\cdot u=0$ implies  $\nabla\cdot v=0$ and tangential vector along $\partial\Omega$ still maps to tangential vector along $\partial\RR^2_+$, though the normal vector is not preserved. Thus, if we want to keep both the divergence-free property and the free boundary condition, we may think of reflection over curved boundary or generalize the system \eqref{LCF} to non-Euclidean metric setup, both will produce extra low order terms in the system.

\bigskip

\section*{Acknowledgements}
Y. Sire is partially support by NSF DMS Grant 2154219, ``Regularity vs singularity formation in elliptic and parabolic equations''.

\bigskip

\end{document}